\documentclass[a4paper,11pt]{amsart}

\usepackage[colorinlistoftodos,bordercolor=orange,backgroundcolor=orange!20,linecolor=orange,textsize=normalsize]{todonotes}
\usepackage{bm}
\usepackage{hyperref}
\usepackage{mathrsfs}
\usepackage{mathtools} 
\usepackage{paralist}
\usepackage{fixmath}
\usepackage{fullpage}

\usepackage[capitalize]{cleveref}

\Crefname{figure}{Figure}{Figures}

\crefformat{equation}{\textup{#2(#1)#3}}
\crefrangeformat{equation}{\textup{#3(#1)#4--#5(#2)#6}}
\crefmultiformat{equation}{\textup{#2(#1)#3}}{ and \textup{#2(#1)#3}}
{, \textup{#2(#1)#3}}{, and \textup{#2(#1)#3}}
\crefrangemultiformat{equation}{\textup{#3(#1)#4--#5(#2)#6}}%
{ and \textup{#3(#1)#4--#5(#2)#6}}{, \textup{#3(#1)#4--#5(#2)#6}}{, and \textup{#3(#1)#4--#5(#2)#6}}

\Crefformat{equation}{Equation~#2\textup{(#1)}#3}
\Crefrangeformat{equation}{Equations~\textup{#3(#1)#4--#5(#2)#6}}
\Crefmultiformat{equation}{Equations~\textup{#2(#1)#3}}{ and \textup{#2(#1)#3}}
{, \textup{#2(#1)#3}}{, and \textup{#2(#1)#3}}
\Crefrangemultiformat{equation}{Equations~\textup{#3(#1)#4--#5(#2)#6}}%
{ and \textup{#3(#1)#4--#5(#2)#6}}{, \textup{#3(#1)#4--#5(#2)#6}}{, and \textup{#3(#1)#4--#5(#2)#6}}

\usetikzlibrary{arrows}
\usetikzlibrary{patterns}
\usetikzlibrary{calc}
\usetikzlibrary{shapes}
\usetikzlibrary{positioning}

\usepackage{amssymb}

\usepackage[scr=boondox,scrscaled=1.05]{mathalfa}

\renewcommand{\geq}{\geqslant}
\renewcommand{\leq}{\leqslant}
\renewcommand{\succeq}{\succcurlyeq}
\renewcommand{\le}{\leq}
\renewcommand{\ge}{\geq}

\newcommand{\arxiv}[1]{\href{http://arxiv.org/abs/#1}{arXiv:#1}}

\DeclareMathOperator{\sign}{\mathsf{sign}}

\DeclareMathAlphabet{\mathbfcal}{OMS}{cmsy}{b}{n}
\DeclareMathAlphabet{\mathbbold}{U}{bbold}{m}{n}

\newcommand{\Z}{\mathbb{Z}}
\newcommand{\R}{\mathbb{R}}    
\newcommand{\Q}{\mathbb{Q}}
\newcommand{\Rinfty}{\R\cup\{-\infty\}}
\newcommand{\puiseux}{\mathbb{K}}
\newcommand{\pospuiseux}{\mathbb{K}_{>0}}
\newcommand{\nnpuiseux}{\mathbb{K}_{\geq 0}}

\newcommand{\trop}[1][]{\ifthenelse{\equal{#1}{}}{ \mathbb{T} }{ \mathbb{T}(#1) }}
\newcommand{\strop}[1][]{{\trop[#1]}_{\pm}}    
\newcommand{\postrop}[1][]{{\trop[#1]}_{+}}    
\newcommand{\negtrop}[1][]{{\trop[#1]}_{-}}    
\newcommand{\barx}{\bar{x}}
\newcommand{\tplus}{\oplus}  
\newcommand{\tsum}{\bigoplus}
\newcommand{\tdot}{\odot}
\newcommand{\tminus}{\ominus}  
\newcommand{\abs}[1]{|{#1}|}

\newcommand{\zero}{{-\infty}}

\DeclareMathOperator*{\val}{\mathsf{val}}
\DeclareMathOperator*{\lc}{\mathsf{lc}}
\DeclareMathOperator*{\sval}{\mathsf{sval}}

\newcommand{\bo}[1]{\mathbold{#1}}

\DeclareMathOperator{\tropPol}{\mathrm{trop}}
\newcommand{\card}[1]{|{#1}|}

\newcommand{\spectra}{\mathcal{S}}
\newcommand{\bspectra}{\mathbfcal{S}}
\newcommand{\bsalg}{\mathbfcal{S}}
\newcommand{\inbspectra}{\bspectra^{\mathrm{in}}}
\newcommand{\outbspectra}{\bspectra^{\mathrm{out}}}
\newcommand{\interspectra}{\mathcal{T}}

\newcommand{\hombspectra}{\bspectra^{h}}

\newcommand{\salg}{\mathscr{S}}
\newcommand{\positivesupport}{\mathscr{S}^{\geq}}
\newcommand{\strictlypositivesupport}{\mathscr{S}^{>}}
\newcommand{\slin}{\salg}
\newcommand{\definable}{\mathcal{S}}

\DeclareMathOperator*{\biglor}{\bigvee}

\newcommand{\limplies}{\rightarrow}
\newcommand{\liff}{\leftrightarrow}
\newcommand{\theory}{\mathrm{Th}}

\newcommand{\lang}{\mathcal{L}}
\newcommand{\logroups}{\lang_{\mathrm{og}}}
\newcommand{\logb}{\lang_{\mathrm{ogb}}}
\newcommand{\lorings}{\lang_{\mathrm{or}}}
\newcommand{\lvrcf}{\lang_{\mathrm{rcvf}}}
\newcommand{\angular}{\mathsf{ac}}
\newcommand{\xsec}{\mathsf{csec}}

\newcommand{\lstructure}{\mathcal{M}}
\newcommand{\universe}{M}

\newtheorem{theorem}{Theorem}[section]
\newtheorem{proposition}[theorem]{Proposition}
\newtheorem{corollary}[theorem]{Corollary}
\newtheorem{lemma}[theorem]{Lemma}
\theoremstyle{definition}
\newtheorem{definition}[theorem]{Definition}
\newtheorem{assumption}[theorem]{Assumption}
\theoremstyle{remark}
\newtheorem{remark}[theorem]{Remark}
\newtheorem{example}[theorem]{Example}

\tikzset{grid/.style={gray!30,very thin}}
\tikzset{axis/.style={gray!50,->,>=stealth'}}
\tikzset{convex/.style={draw=none,fill=lightgray,fill opacity=0.7}}
\tikzset{convexborder/.style={very thick}}
\tikzset{point/.style={blue!50}}
\tikzset{hs/.style={fill opacity=0.3,fill=orange,draw=none}}
\tikzset{hsborder/.style={orange,ultra thick,dashdotted}}

\newcommand{\overbar}[1]{\mkern 1.5mu\overline{\mkern-1.5mu#1\mkern-1.5mu}\mkern 1.5mu}

\newcommand{\polyh}{\mathcal{W}}
\newcommand{\polyhII}{\mathcal{V}}
\newcommand{\face}{\polyhII}
\newcommand{\hplane}{\mathcal{H}}
\newcommand{\ri}{\mathrm{ri}}
\newcommand{\conv}{\mathrm{conv}}

\newcommand{\domin}{\mathrm{Argmax}}

\newcommand{\psupport}{\Lambda}
\newcommand{\subpsupport}{L}
\newcommand{\complex}{\mathcal{C}}
\newcommand{\poscomplex}{\complex^{\geq}}

\newcommand{\loew}{\succeq}
\newcommand{\inter}{\mathrm{int}}
\DeclareMathOperator{\vecspan}{span}
\DeclareMathOperator{\cl}{cl}

\newcommand{\vfield}{\mathscr{K}}
\newcommand{\vgroup}{\Gamma}
\newcommand{\resfield}{\mathscr{k}}
\newcommand{\vring}{\mathscr{O}}
\newcommand{\videal}{\mathscr{M}}
\newcommand{\res}{\mathsf{res}}

\newcommand{\rcfield}{\vfield}

\newcommand{\oag}{\Gamma}
\newcommand{\doag}{\Gamma}

\newcommand{\thVrcf}{\theory_{\mathrm{rcvf}}}
\newcommand{\thRcf}{\theory_{\mathrm{rcf}}}
\newcommand{\thDoagb}{\theory_{\mathrm{doagb}}}

\newcommand{\dgraph}{\vec{\mathcal{G}}}
\newcommand{\vertices}{V}
\newcommand{\vertex}{v}
\newcommand{\edges}{E}
\newcommand{\edge}{e}
\newcommand{\tails}{T}
\newcommand{\head}{h}
\newcommand{\inedge}{\mathrm{In}}
\newcommand{\outedge}{\mathrm{Out}}
\newcommand{\circulation}{\gamma}

\newcommand{\stbase}{\epsilon}
\newcommand{\diag}{\mathrm{tdiag}}

\newcommand{\minspectra}{\spectra^{\operatorname{form}}}

\newcommand{\torth}{\trop_{\pm,\mu}}

\begin{document}

\title{Tropical Spectrahedra}
\date{February 18, 2020}

\thanks{The three authors were partially supported by the ANR projects CAFEIN (ANR-12-INSE-0007) and MALTHY (ANR-13-INSE-0003), by the PGMO program of EDF and Fondation Math\'ematique Jacques Hadamard, and by
the ``Investissement d'avenir'', r{\'e}f{\'e}rence ANR-11-LABX-0056-LMH,
LabEx LMH. M.~Skomra was supported by a grant from R{\'e}gion Ile-de-France.}

\author{Xavier Allamigeon}
\author{St{\'e}phane Gaubert}
\address{INRIA and CMAP, \'Ecole Polytechnique, IP Paris, CNRS, 91128 Palaiseau Cedex, France}
\email{firstname.lastname@inria.fr}

\author{Mateusz Skomra}
\address{Univ Lyon, EnsL, UCBL, CNRS,  LIP, F-69342, LYON Cedex 07, France}\email{mateusz.skomra@ens-lyon.fr} 

\catcode`\@=11
\@namedef{subjclassname@2010}{%
  \textup{2010} Mathematics Subject Classification}
\catcode`\@=12

\begin{abstract}
We introduce tropical spectrahedra,
defined as the images by the nonarchimedean valuation
of spectrahedra over the field of real Puiseux series.
We provide an explicit polyhedral characterization of generic tropical spectrahedra,
involving principal tropical minors of size at most~$2$. One of the key ingredients is Denef--Pas quantifier elimination result  over valued fields. We obtain from this that the nonarchimedean valuation maps semialgebraic sets to semilinear sets that are closed. 
We also prove that, under a regularity assumption, the image by the valuation of a  basic semialgebraic set is obtained by tropicalizing the inequalities which define it.
\end{abstract}

\maketitle

\section{Introduction}

Spectrahedra are one of the main generalizations of polyhedra. They are
real convex semialgebraic sets, defined by a single matrix inequality 
of the form 
\[ Q^{(0)}+
x_{1}Q^{(1)}+\dots + x_{n}Q^{(n)} \loew 0 \enspace ,
\]
where $Q^{(0)},\dots,Q^{(n)}$ are
real symmetric matrices, and $\loew$ denotes the Loewner order.
(Recall that, for real symmetric matrices $A,B$, the {\em Loewner order}
is such that $A\loew B$ if all the eigenvalues of $A-B$ are nonnegative.) Spectrahedra arise in a number of applications from engineering
sciences and combinatorial optimization~\cite{siam_parrilo,matousekbook}. 
Several theoretical questions
concerning spectrahedra have been raised, such as the 
Helton--Nie conjecture~\cite{helton_nie},
recently disproved by Scheiderer~\cite{scheiderer_helton_nie}.
Other, like the generalized Lax conjecture~\cite{vinnikov} or the complexity of checking the emptiness~\cite{ramana_exact_duality}, are still unsettled.

Spectrahedra can be considered more generally over any real closed field.
In tropical geometry, one often works with the usual field of real Puiseux series (with rational exponents), 
or
with a larger field
of generalized Puiseux series with real exponents,
as we do here. 

\subsection*{Main results} In this paper, we introduce tropical spectrahedra. These are defined as the images
by the nonarchimedean valuation of spectrahedra over the field of real generalized Puiseux series. Our main result (\cref{corollary:genericity}) provides an explicit polyhedral characterization of tropical spectrahedra,
when the valuation of the defining matrices satisfy a genericity condition.
We suppose first that
the defining matrices have a Metzler type sign pattern.
Metzler matrices are fundamental in mathematical modelling,
since they yield dynamical systems whose flow is order preserving.
However, we focus on these matrices because
the tropicalization of the (generic) spectrahedra
that they define is remarkably transparent (\cref{corollary:metzler_regularity}). Then, we show how to recover the tropicalization
of a spectrahedron whose defining matrices have a general
sign pattern by ``gluing'' tropicalizations of spectrahedra
satisfying the Metzler condition.
Our results are synthesized by \cref{concide_minors},
which shows, still {\em under a genericity condition}, that
{\em the tropicalization of a spectrahedron is obtained
by requiring only the positivity of principal tropical minors
of order at most $2$}. 
The genericity condition is expressed explicitly as a flow condition
in a directed hypergraph (\cref{lemma:metzler_genericity} and \cref{theorem:genericity}).

To this end, we study, more generally,
tropical semialgebraic sets, defined as the images
by the nonarchimedean valuation of semialgebraic sets
over the field of real generalized Puiseux series. We exploit 
quantifier elimination methods in valued fields by Denef~\cite{denef_p-adic_semialgebraic} and Pas~\cite{pas_cell_decomposition},
which imply that such sets are semilinear. Moreover, we show that
tropical semialgebraic sets are always closed. It follows that,
under a regularity assumption, 
the image by the valuation of a basic semialgebraic set is obtained
by ``tropicalizing'' each polynomial inequality arising in the definition
of this set. This shows in particular that this image is a polyhedral complex
with an explicit description in terms of piecewise linear inequalities
(\cref{corollary:purity_commutation}).

\subsection*{Related work}
A general question, in tropical geometry, consists in providing
combinatorial characterizations of nonarchimedean amoebas, i.e.,
images by the nonarchimedean valuation of algebraic sets
over nonarchimedean fields. 
Kapranov's theorem on amoebas of hypersurfaces, 
or Viro's patchworking method for real algebraic curves~\cite{viro}
and its extensions~\cite{sturmfels_complete_intersection,bihan_complete_intersection}, address this question in different settings.
In parallel, general results have been developed in model theory of valued
fields, in particular by Weispfenning~\cite{weispfenning_qe_in_valued_fields}, Denef~\cite{denef_rationality_of_poincare,denef_p-adic_semialgebraic}, and Pas~\cite{pas_cell_decomposition,pas_mixed_characteristic}.
The fact that nonarchimedean amoebas have a polyhedral structure follows
from these works. 

Excepting tropical polyhedra, there are few works dealing with tropical
semialgebraic sets. The most closely related works are those
of Yu~\cite{yu_semidefinite_cone} and Alessandrini~\cite{alessandrini}.

Yu characterized the image by the nonarchimedean valuation
of the cone of positive semidefinite matrices over real Puiseux series,
showing that it is determined by $2\times 2$ principal tropical minors. 
We show that $2\times 2$, together with $1\times 1$, tropical minors
still determine generic tropical spectrahedra.

Alessandrini studied the log-limits of real semialgebraic
sets. His approach, based on o-minimal models, shows in particular
that the image by the nonarchimedean valuation
of a semialgebraic set over the field of absolutely convergent real generalized
Puiseux series is a polyhedral complex. We avoid the recourse to o-minimal techniques by using Denef--Pas quantifier elimination instead. A by-product is that the property holds for any real closed valued field.

\subsection*{Applications and consequences of the main results}
A general issue in computational optimization is to develop combinatorial algorithms for semidefinite programming. The present work, providing an explicit characterization
of tropical spectrahedra, leads to combinatorial algorithms
to solve a class of generic semidefinite feasibility problems over nonarchimedean fields.
This is developed in the companion work~\cite{issac2016jsc}, where it is shown that feasibility
problems for generic
tropical spectrahedra 
are equivalent to solving mean payoff stochastic games with perfect information, also known as ``turn-based'' mean payoff stochastic games.
This allows one to apply game algorithms to solve nonarchimedean semidefinite
feasibility problems.
The reference~\cite{issac2016jsc} focuses on algorithmic aspects, relying on the present work for
structural results.
Conversely, tropical spectrahedra offer new methods
to study turn-based mean payoff stochastic games. Recall
that the complexity of these games is one of the fundamental unsettled problems
in algorithmic game theory: these games are in NP $\cap$ coNP, but they are not known to be in P (see~\cite{andersson_miltersen} and the references therein). In the recent joint work of the authors with Katz~\cite{1802.07712}, we associate ``primal'' and ``dual'' tropical spectrahedra to a turn-based mean payoff stochastic game. We show that metric characteristics of these spectrahedra define a condition number, allowing one
to parametrize the complexity of these games. 

The present work has also more theoretical consequences. In~\cite{mega2017}, 
we use the characterization of tropical spectrahedra
to show that an analogue of the Helton--Nie conjecture over nonarchimedean fields is true ``up to taking the valuations'', i.e.,
the images by the valuation of convex semialgebraic sets over a nonarchimedean
field coincide with the images by the valuation of projected spectrahedra, over
the same field.

\subsection*{Organization of the paper}

In \cref{sec:prelim}, we recall basic notions and results from tropical
geometry and from the theory of valued fields. In \cref{sec:model}, we apply
the quantifier elimination results of Denef and Pas to show that
tropical semialgebraic sets have a polyhedral structure. This allows
us to show, in \cref{sec:images}, that tropical semialgebraic sets are finite
unions of closed polyhedra. In \cref{section:spectrahedra},
we introduce tropical spectrahedra.
We first provide an explicit combinatorial description in the
simpler situation in which the input matrices have a Metzler
sign pattern (\cref{subsubsec:tropical_metzler_spectrahedra}),
and subsequently remove this assumption (\cref{subsection:nonmetzler}).
These results hold under a condition that is shown 
to be satisfied generically in \cref{section:generic}.

\section{Preliminaries}\label[section]{sec:prelim}

\subsection{Puiseux series}\label[section]{sec:puiseux}

The main example of nonarchimedean structure which we use in this paper is the \emph{field $\puiseux$ of (formal generalized real) Puiseux series}. This field is defined as follows. It consists of formal series of the form
\begin{equation}
\bo x = \sum_{i = 1}^{\infty} c_{\lambda_{i}}t^{\lambda_{i}} \, , \label[equation]{eq:series}
\end{equation}
where $t$ is a formal parameter, $(\lambda_{i})_{i \ge 1}$ is a strictly decreasing sequence of real numbers that is either finite or unbounded, and $c_{\lambda_{i}} \in \R \setminus \{ 0\}$ for all $\lambda_{i}$. There is also a special, empty series, which is denoted by $0$. 
We denote by $\lc(\bo x)$ the coefficient $c_{\lambda_{1}}$ of the leading term in the series $\bo x$, with the convention that $\lc(0) = 0$. The addition and multiplication in $\puiseux$ are defined in a natural way. Moreover, $\puiseux$ can be endowed with a linear order $\ge$, which is defined as $\bo x \ge \bo y$ if $\lc(\bo x - \bo y) \ge 0$.  We denote $\nnpuiseux$ the set of nonnegative series $\bo x$, i.e.,~satisfying $\bo x \geq 0$. The \emph{valuation} of an element $\bm x \in \puiseux$ as in~\cref{eq:series} is defined as the greatest exponent $\lambda_{1}$ occurring in the series. It is known that $\puiseux$ is a real closed field (see~\cite{markwig} for instance). We use the specific field $\puiseux$ to keep the exposition concrete. 
Our main results, however, work for any real closed field sent surjectively to $\R$ by a nonarchimedean valuation. 
Other convenient choices of such fields include the field of Hahn series (series with well-ordered support) with real coefficients and real exponents~\cite{ribenboim}, or its subfield consisting of those series that are absolutely convergent for $t$ large enough, or the subfield
of $\puiseux$ consisting of series subject to the same convergence requirement (see~\cite{van_dries_power_series} for more information on the two latter fields).

\subsection{Tropical algebra}\label[section]{section:tropical_algebra}

In this section, we briefly introduce the basic concepts of tropical algebra and its connection with the nonarchimedean field of Puiseux series.

We denote by $\val \colon \puiseux \to \Rinfty$ the function which maps a Puiseux series $\bo x \in \puiseux$ to its valuation. We use the convention $\val(0) = -\infty$.  It is immediate to see that the map $\val$ satisfies the following properties
\begin{align}
\val(\bo x + \bo y) &\leq \max( \val (\bo x ), \val (\bo y))\label[equation]{e-val}\\
\val(\bo x  \bo y) &= \val (\bo x ) + \val (\bo y)\label[equation]{e-val2}
\end{align}
meaning that $\val$ is a {\em nonarchimedean valuation}. Moreover, the equality
holds in~\cref{e-val} if the leading terms of $\bo x$ and $\bo y$ do not cancel, which is the case if $\val (\bo x ) \neq \val (\bo y)$ or if
$\bo x, \bo y\geq 0$.

Loosely speaking, the \emph{tropical semifield} $\trop$ can be thought of as the
image of $\puiseux$ by the nonarchimedean valuation.
We use the ``max-plus'' convention,
so the base set of $\trop$
is defined to be $\Rinfty$. It is endowed with the addition
$x \tplus y \coloneqq \max(x,y)$ and the multiplication $x \tdot y \coloneqq x + y$. 
The term ``semifield'' refers to the fact that the addition does not have
an opposite law.
We use the notation $\tsum_{i = 1}^{n}a_{i} = a_{1} \tplus \cdots \tplus a_{n}$ and $a^{\tdot n} = a \tdot \cdots \tdot a$ ($n$ times). We also endow $\trop$ with the standard order $\ge$. 
The map $\val$ yields an order-preserving
morphism of semifields from $\nnpuiseux$ to $\trop$. This follows
from~\cref{e-val2} and from the equality case in~\cref{e-val}.
We refer the reader to~\cite{butkovic,maclagan_sturmfels} for more information on the tropical semifield.

We note that in the valued fields literature, the valuation map is usually defined using the min-plus notation. More precisely, in the definition of Puiseux series,
one usually requires the sequence of exponents $(\lambda_{i})_{i \ge 1}$
to be strictly increasing, and then the valuation of a series
is defined as the smallest exponent occurring in this series.
This leads to a map $\val \colon \puiseux \to \R \cup \{+\infty\}$ that satisfies $\val(x_{1} + x_{2}) \ge \min(\val(x_{1}),\val(x_{2}))$. This definition is equivalent to the one above---it is enough to replace $\val(x)$ by $-\val(x)$ in order to alternate between the two definitions.
In this paper, we prefer to use the max-plus notation, since
it yields an order preserving map from $\nnpuiseux$ to $\trop$,
allowing more direct arguments. Our convention also corresponds
to an analytic interpretation of $t$ as a ``large positive parameter'',
and it is consistent with earlier work of us on related problems,
see~\cite{issac2016jsc,1708.01544}.

It is convenient to keep track not only of the valuation of a series but also of its sign. To this end, we define the \emph{sign} of a series $\bo x \in \puiseux$ as $+1$ if $\bo x >0$, $-1$ if $\bo x<0$, and $0$ otherwise. We denote it by $\sign (\bo x)$. Besides, we introduce the \emph{signed valuation}, denoted by $\sval$,
which associates the couple
$(\sign (\bo x), \val (\bo x))$ with a series $\bo x\in\puiseux$. We denote by $\strop$ the
image of $\puiseux$ by $\sval$. We refer to it as the set
of \emph{signed tropical numbers}. For brevity,
we denote an element of the form $(\epsilon,a)$ by $a$
if $\epsilon =1$, $\tminus a$ if $\epsilon =-1$,
and $-\infty$ if $\epsilon =0$. Here, $\tminus$ is a formal
symbol. We call the elements of the first and second kind
the \emph{positive} and \emph{negative} tropical numbers, respectively.
We denote by $\postrop$ and $\negtrop$ the corresponding sets.
In this way, $(-2)$ is tropically positive, but $\tminus (-2)$ is tropically negative. Also, $\trop$ is embedded in $\strop$, i.e., $\trop = \postrop \cup \{ \zero\}$. We shall extend the valuation maps $\val$ and $\sval$ to vectors and matrices in a coordinate-wise manner.

In $\strop$, we define a \emph{modulus} function, $\abs{\cdot} \colon \strop \to \trop$, as $\abs{\zero} = \zero$ and $\abs{a} = \abs{\tminus a} = a$ for all $a \in \postrop$. We point out that $\tdot$ straightforwardly extends
to $\strop$ using the standard rules for the sign, for instance $2\tdot (\tminus 3)= \tminus 5$. 
In contrast, we only partially define the tropical addition $\tplus$ to elements of $\strop$ of identical sign, e.g., $2 \tplus 3 = 3$ and $(\tminus 2) \tplus (\tminus 3) = \tminus 3$. Such a partially defined operation
will be enough for the present purposes.
We note, however, that we may extend the addition operation so that
it is everywhere defined, by embedding $\strop$ in the {\em symmetrized tropical semiring}~\cite{guterman}, or by allowing the addition to be multivalued
in the setting of hyperfields~\cite{virohyperfields,connesconsani,baker}. 
 
Moreover, we use the notion of tropical polynomials. A \emph{(signed) tropical polynomial} over the variables $X_1, \dots, X_n$ is a formal expression of the form
\begin{align}
P(X) = \tsum_{\alpha \in \psupport} a_{\alpha} \tdot X_{1}^{\tdot \alpha_{1}} \tdot  \cdots \tdot X_{n}^{\tdot \alpha_{n}} \, ,\label[equation]{e-def-P}
\end{align}
where $\psupport$ is a finite subset of $\{0, 1, 2, \dots \}^{n}$, and $a_{\alpha} \in \strop \setminus \{\zero\}$ for all $\alpha \in \psupport$. 
We set $\psupport^{+}\coloneqq \{\alpha \in \psupport\colon a_\alpha \in \postrop\}$
and $\psupport^{-}\coloneqq \{\alpha \in \psupport\colon a_\alpha \in \negtrop\}$. We shall occasionally write $\psupport(P)$ or $\psupport^\pm(P)$
to emphasize the dependence in $P$.
We say that the tropical polynomial $P$ \emph{vanishes} on the point $x\in \strop^n$ if two of the terms $a_{\alpha} \tdot x_{1}^{\tdot \alpha_{1}} \tdot  \dots \tdot x_{n}^{\tdot \alpha_{n}}$ which have the greatest modulus do not have the same sign. If $P$ does not vanish on $x$, we define $P(x)$ as the tropical sum of the terms which have the greatest modulus. As an example, if $P(X)=2\tdot X_1^{\tdot 3}\tdot X_2^{\tdot 4}\tplus (\tminus 0 \tdot X_2)$, 
then $P(1,\tminus 5)=25$, $P(1,-5)= \tminus (-5)$, whereas 
$P$ vanishes on $(1,-5/3)$. 

Furthermore, if $\psupport$ is empty, then we define $P(x) = -\infty$ for all $x \in \strop^{n}$. The next definition and lemma relate the structure laws of $\strop$ and the ones of $\puiseux$.

\begin{definition}\label[definition]{def-formaltrop}
If 
\begin{equation}\label{eq:puis_pol}
\bo P(X)= \sum_{\alpha \in \psupport} \bo a_{\alpha} X_{1}^{\alpha_{1}} \dots
X_{n}^{\alpha_{n}} \in \puiseux[X_1,\dots,X_n]
\end{equation}
is a polynomial over Puiseux series, then we define its \emph{formal tropicalization}, denoted $\tropPol(\bo P)$, as the tropical polynomial
\[ 
\tropPol(\bo P) \coloneqq  \tsum_{\alpha \in \psupport} \sval (\bo a_\alpha) \tdot X_{1}^{\tdot \alpha_{1}} \tdot  \cdots \tdot X_{n}^{\tdot \alpha_{n}} \, .
\]
\end{definition}

\begin{lemma}\label[lemma]{lemma:homomorphism}
Let $\bo P(X) \in \puiseux[X_1,\dots,X_n]$ and let $P \coloneqq \tropPol(\bo P)$. Then, for all $\bo x\in \puiseux^n$, $\sval (\bo P(\bo x)) = P(\sval (\bo x))$ provided that $P$ does not vanish on $\sval (\bo x)$.
\end{lemma}
\begin{proof}
Let $\bo P$ be as in~\cref{eq:puis_pol}. If $\bo P$ is a zero polynomial, then the claim is trivial. Otherwise, let $x \coloneqq \sval(\bo x)$ and $a_{\alpha} \coloneqq \sval(\bo a_{\alpha})$ for all $\alpha \in \psupport$. Note that for every $\alpha \in \psupport$ we have the equalities
\begin{equation*}\label{eq:trop_by_terms}
\begin{aligned}
\sval(\bo a_{\alpha} \bo x_{1}^{\alpha_{1}} \dots \bo x_{n}^{\alpha_{n}}) &= a_\alpha \tdot x_{1}^{\tdot \alpha_{1}} \tdot  \dots \tdot x_{n}^{\tdot \alpha_{n}} \, , \\
\val(\bo a_{\alpha} \bo x_{1}^{\alpha_{1}} \dots \bo x_{n}^{\alpha_{n}}) &= \abs{a_\alpha \tdot x_{1}^{\tdot \alpha_{1}} \tdot  \dots \tdot x_{n}^{\tdot \alpha_{n}}} \, .
\end{aligned}
\end{equation*} 
Let $\overbar{\psupport} \subset \psupport$ denote terms that maximize the modulus of $a_\alpha \tdot x_{1}^{\tdot \alpha_{1}} \tdot  \dots \tdot x_{n}^{\tdot \alpha_{n}}$. Write $\bo P( \bo x)$ as
\[
\bo P( \bo x) = \Bigl(\sum_{\alpha \in \overbar{\psupport}} \bo a_{\alpha} \bo x_{1}^{\alpha_{1}} \dots \bo x_{n}^{\alpha_{n}}\Bigr) + \Bigl(\sum_{\alpha \notin \overbar{\psupport}} \bo a_{\alpha} \bo x_{1}^{\alpha_{1}} \dots \bo x_{n}^{\alpha_{n}}\Bigr) \, .
\]
Since $P$ does not vanish on $\sval (\bo x)$, the terms $\bo a_{\alpha} \bo x_{1}^{\alpha_{1}} \dots \bo x_{n}^{\alpha_{n}}$ such that $\alpha \in \overbar{\psupport}$ have the same sign. Furthermore, note that $\sval(\bo x + \bo y) = \sval(\bo x) \tplus \sval(\bo y)$ provided that $\bo x$ and $\bo y$ have the same sign. In particular, we have the equality
\begin{align*}
\sval\Bigl(\sum_{\alpha \in \overbar{\psupport}} \bo a_{\alpha} \bo x_{1}^{\alpha_{1}} \dots \bo x_{n}^{\alpha_{n}}\Bigr) &= \tsum_{\alpha \in \overbar{\psupport}} a_{\alpha} \tdot x_{1}^{\tdot \alpha_{1}} \tdot  \cdots \tdot x_{n}^{\tdot \alpha_{n}} = P(x) \, .
\end{align*}
Moreover, observe that $\sval(\bo x + \bo y) = \sval(\bo x)$ whenever $\val(\bo y) < \val(\bo x)$. Hence, we get
\[
\sval(\bo P(\bo x)) = \sval\Bigl(\sum_{\alpha \in \overbar{\psupport}} \bo a_{\alpha} \bo x_{1}^{\alpha_{1}} \dots \bo x_{n}^{\alpha_{n}}\Bigr) = P(x) \, . \qedhere
\]
\end{proof}

Given a polynomial $\bm P$ as in~\cref{eq:puis_pol},
we denote by $\bm P^+$ the polynomial obtained by summing
the terms $\bm a_{\alpha} X_{1}^{\alpha_{1}} \dots
X_{n}^{\alpha_{n}}$ such that $\bm a_\alpha > 0$. Similarly, $\bm P^-$ refers to the
sum of
the terms $-\bm a_{\alpha} X_{1}^{\alpha_{1}} \dots X_{n}^{\alpha_{n}}$ verifying $\bm a_\alpha < 0$. In this way, $\bm P = \bm P^+ - \bm P^-$. We also use the analogues of these polynomials in the tropical setting. If $P$ is the tropical polynomial given in~\cref{e-def-P}, we define $P^+$ (resp.\ $P^-$) as the tropical polynomial obtained by summing 
the terms $\abs{a_\alpha} \tdot X_{1}^{\tdot \alpha_{1}} \tdot  \dots \tdot X_{n}^{\tdot \alpha_{n}}$ where $a_\alpha \in \postrop$ (resp.\ $\negtrop$). Observe that the quantities $P^+(x)$ and $P^-(x)$ are well defined for all $x \in \trop^n$, since the tropical polynomials $P^+$ and $P^-$ only involve tropically positive coefficients.
In particular, if $\bm P$ is as in~\cref{eq:puis_pol}, and if $P=
\tropPol(\bo P)$, observe that $P^+=\tropPol(\bm P^+)$ and
$P^-=\tropPol(\bm P^-)$.

Throughout the paper, we denote the set $\{1, \dots, k\}$ by~$[k]$.

\subsection{Tropical polynomial inequalities and polyhedral complexes}

Given a tropical polynomial $P$ as in~\cref{e-def-P}, we say that $P$ is \emph{nonzero} if the set $\psupport$ is nonempty. For every such tropical polynomial and every point $x \in \R^{n}$ we define the \emph{set of maximizing multi-indices at $x$} as
\[
\domin(P,x) \coloneqq \bigl\{\alpha \in \psupport \colon \forall \beta \in \psupport, \, \abs{a_{\alpha}} + \langle \alpha, x \rangle \geq \abs{a_{\beta}} + \langle \beta, x \rangle\bigr\} \, ,
\]
where $\langle \cdot, \cdot \rangle$ refers to the usual scalar product. If $P$ is a nonzero tropical polynomial and we fix a multi-index $\alpha \in \psupport$, then the set
\[
\polyh = \cl(\{x \in \R^{n} \colon \domin(P, x) = \alpha \}) \, 
\]
is a polyhedron that is either empty or full-dimensional. (Here and in the sequel, $\cl(\cdot)$ refers to the closure of a subset of $\R^n$ with respect to the Euclidean topology.) Moreover, the family of these polyhedra, together with their faces, forms a polyhedral complex whose support is equal to $\R^{n}$. More precisely, a polyhedron $\polyhII$ is a (possibly empty) cell of this complex if and only if there exists a subset $\subpsupport \subset \psupport$ such that
\begin{equation}
\polyhII = \cl(\{x \in \R^{n} \colon \domin(P,x) = \subpsupport \}) \, . \label[equation]{eq:cell}
\end{equation}
We denote this complex by $\complex(P)$. The union of all $(n-1)$-dimensional polyhedra belonging to $\complex(P)$ is called a \emph{tropical hypersurface}. In other words, a tropical hypersurface is the set of all points $x \in \R^{n}$ such that $\domin(P,x)$ has at least two elements. 

For the purpose of this work, given a nonzero tropical polynomial $P$,
it is also convenient to consider the set 
\[
\positivesupport(P)\coloneqq
\{x \in \R^{n} \colon P^{+}(x) \ge P^{-}(x) \} \enspace.
\]
To describe this set, we consider the family $\poscomplex(P)$ of positive cells of $\complex(P)$. We say that a cell $\polyhII \in \complex(P)$ as in \cref{eq:cell} is \emph{positive} if there exists at least one $\alpha \in \subpsupport$ such that $a_{\alpha} \in \postrop$ or if $\polyhII$ is empty. The family $\poscomplex(P)$ is a polyhedral complex whose support is equal to 
$\positivesupport(P)$.

Given a system of nonzero tropical polynomials $P_{1}, \dots, P_{m}$, one can regard the refinements of complexes defined by $P_{1}, \dots, P_{m}$. More precisely, we define $\complex(P_{1}, \dots, P_{m})$ and $\poscomplex(P_{1}, \dots, P_{m})$ as
\begin{equation*}
\begin{aligned}
\complex(P_{1}, \dots, P_{m}) &= \bigl\{ \cap_{i = 1}^{m} \polyh_{i} \colon \forall i, \ \polyh_{i} \in \complex(P_{i}) \bigr\} \, , \\
\poscomplex(P_{1}, \dots, P_{m}) &= \bigl\{ \cap_{i = 1}^{m} \polyh_{i} \colon \forall i, \ \polyh_{i} \in \poscomplex(P_{i}) \bigr\} \, .
\end{aligned}
\end{equation*}
The families $\complex(P_{1}, \dots, P_{m})$ and $\poscomplex(P_{1}, \dots, P_{m})$ are polyhedral complexes. The support of the former is equal to $\R^{n}$ while the support of the latter coincides with
\[
\positivesupport(P_1,\dots,P_m)\coloneqq
\{x \in \R^{n} \colon \forall i, P_{i}^{+}(x) \ge P_{i}^{-}(x) \}
\enspace .
\]
Finally, in this work we consider polyhedral complexes with regular supports. Recall that a closed set $S \subset \R^{n}$ is called \emph{regular} if
$S = \cl(\inter(S))$ (here and in the sequel, $\inter(\cdot)$ denotes the interior of a subset of $\R^n$). If $\complex$ is a polyhedral complex, then its support is regular if and only if $\complex$ is pure and full-dimensional. 
A basic property of such complexes appears in the next lemma.
\begin{lemma}\label[lemma]{lemma:purity_closes_strict}
Suppose that 
the polyhedral complex $\poscomplex(P_{1}, \dots, P_{m})$ 
has a regular support.
Then 
this support,
$\positivesupport(P_1,\dots,P_m)$,
coincides with the closure of the set
\[
\strictlypositivesupport(P_1,\dots,P_m)\coloneqq\{x \in \R^{n} \colon \forall i, P_{i}^{+}(x) > P_{i}^{-}(x) \} \, .
\]
\end{lemma}
\begin{proof}
The set $\positivesupport(P_1,\dots,P_m)$ is closed 
since the tropical polynomial functions $P_i^{\pm}$ are continuous,
and obviously, $\strictlypositivesupport(P_1,\dots,P_m)\subset
\positivesupport(P_1,\dots,P_m)$. Therefore,
\[ 
\cl(\strictlypositivesupport(P_1,\dots,P_m))\subset \positivesupport(P_1,\dots,P_m) \, .
\]
Consider now $y\in \positivesupport(P_1,\dots,P_m)$. Since
this set is regular, $y$ belongs to a full-dimensional cell $\polyh$
of $\poscomplex(P_1,\dots,P_m)$. We have $\polyh= \cap_{1\leq i\leq m}\polyh_i$,
where $\polyh_i$ is a full-dimensional cell of $\poscomplex(P_i)$.
This implies that $\polyh_i= \cl(\{x \in \R^{n} \colon \domin(P_i,x) = \subpsupport_i \})$ where $\subpsupport_i$ is a one element
subset of $\psupport^+(P_i)$. We conclude that $P^+_i(x)>P^-_i(x)$
holds for all $x\in \inter(\polyh_i)$, and so, $\inter(\polyh)\subset
\strictlypositivesupport(P_1,\dots,P_m)$.
Taking any $\bar{y}$ in the interior of $\polyh$, we see that the half-open segment $[\bar{y},y[$ is contained in the interior
of $\polyh$, and thus in $\strictlypositivesupport(P_1,\dots,P_m)$.
It follows that $y\in \cl(\strictlypositivesupport(P_1,\dots,P_m))$.
\end{proof}

\subsection{Valued fields}\label[section]{sec:valued_fields}

In this section, we recall some basic information about valued fields. We refer to~\cite[Chapter~2]{engler_prestel_valued_fields} for a complete account. If $\vfield$ is a field and $\vgroup$ is an ordered abelian group, then a surjective function $\val \colon \vfield \to \vgroup \cup \{\zero \}$ is called a \emph{valuation} if it fulfills the following three conditions:
\begin{equation}
\begin{aligned}
 \val(x) = -\infty &\iff x = 0 \, , \\
\forall x_{1}, x_{2} \in \vfield, \ \val(x_{1}x_{2}) &= \val(x_{1}) + \val(x_{2}) \, , \\
\forall x_{1}, x_{2} \in \vfield, \ \val(x_{1} + x_{2}) &\le \max(\val(x_{1}),\val(x_{2})) \, . \label[equation]{eq:valuation}
\end{aligned}
\end{equation}
A tuple $(\vfield, \vgroup, \val)$ is called a \emph{valued field}. Under these conditions, $\vring \coloneqq \{x \in \vfield \colon \val(x) \le 0 \}$ is a subring of $\vfield$ and $\videal \coloneqq \{ x \in \vfield \colon \val(x) < 0\}$ is its maximal ideal. The quotient field $\resfield \coloneqq \vring/\videal$ is called the \emph{residue field}. We denote by $\res$ the canonical projection from $\vring$ to $\resfield$. The valuation is called \emph{trivial} if $\vgroup = \{0\}$. Otherwise, it is called \emph{nontrivial}.
Recall that, as explained in~\cref{section:tropical_algebra}, we use
the ``max-plus'' sign-convention to define the valuation, i.e.,
the map $-\val$
is a valuation in the sense of~\cite{engler_prestel_valued_fields}.

A map $\xsec \colon \vgroup \to \vfield^{*}$ is called a \emph{cross-section} 
if it is a multiplicative morphism such that $\val \circ \xsec$ is the identity map. A map $\angular \colon \vfield \to \resfield$ is called an \emph{angular component} if it fulfills the following conditions:
\begin{itemize}
\item $\angular(0) = 0$;
\item $\angular$ is a multiplicative morphism from $\vfield^{*}$ to $\resfield^{*}$;
\item the function from $\vring$ to $\resfield$, mapping $x$ to $\angular(x)$ if $\val(x) = 0$, and to $0$ otherwise, is a surjective morphism of rings whose kernel is equal to $\videal$. 
\end{itemize}
Not every valued field admits an angular component~\cite{pas_on_angular_component}. Nevertheless, if it admits a cross-section $\xsec$, then $\angular(x) \coloneqq \res(\xsec(-\val(x))x)$ for $x \neq 0$ defines an angular component. For example, if $\puiseux$ is a field of Puiseux series defined in \cref{sec:puiseux}, then $\xsec(y) = t^{y}$ is a cross-section, and $\lc$ is an angular component. In fact, every real closed valued field has a cross-section, as shown by the following lemma.
\begin{lemma}
Suppose that $\vfield$ is real closed. Then $(\vfield, \vgroup, \val)$ admits a cross-section. (In particular, it has an angular component.)
\end{lemma}
\begin{proof}
The case where the valuation is trivial is obtained by taking a cross-section equal to $1$. Therefore, we assume that the valuation is nontrivial. First, observe that in this case $\vgroup$ is a divisible group, as any positive element of $\vfield$ admits an $n$th root for every nonzero natural number $n$. 
Moreover, since $\vgroup$ is ordered, it is also torsion free. Therefore, given a nonzero natural number $n$ and $y \in \Gamma$, the equation $n z = y$ has a unique solution $z$ in $\Gamma$. It follows that we can regard $\vgroup$ as a vector space over $\Q$. Let $\{y_{i} \}_{i \in I}$ be a basis of this space. For every $i$ take $x_{i} \in \vfield$ such that $x_{i} > 0$ and $\val(x_{i}) = y_{i}$. For every finite subset $J \subset I$ and every $(\alpha_{j}) \in \Q^{J}$ define
\[
\xsec(\sum_{j \in J} \alpha_{j} y_{j}) = \prod_{j \in J} x_{j}^{\alpha_{j}} \, .
\]
It is obvious that $\xsec$ is a cross-section.
\end{proof}

Finally, we recall the notion of a convex valuation. Suppose that $\vfield$ is an ordered field with a total order $\ge$. We say that the valuation $\val$ is \emph{convex with respect to $\ge$} if it satisfies the following property: for every $x_{1} \in \vring$ and every $x_{2} \in \vfield$ we have the implication
\[
0 \le x_{2} \le x_{1} \implies x_{2} \in \vring \, .
\]
If $\vfield$ is a real closed field, it has a unique total order. In this case, the convexity property is understood in the sense of this order. It can be shown that if $\vfield$ is real closed and $\val$ is convex, then $\resfield$ is also real closed (see \cite[Theorem~4.3.7]{engler_prestel_valued_fields}). The field of Puiseux series is an example of a real closed field with convex valuation.

\section{Semilinearity of tropical semialgebraic sets}\label[section]{sec:model}

The goal of this section is to prove the following theorem.
\begin{theorem}\label[theorem]{theorem:image_finite_union_of_ri_polyh}
Let $\vfield$ be a real closed field equipped with a nontrivial and convex valuation~$\val$. Furthermore, suppose that the set $\salg \subset \vfield^{n}$ is semialgebraic. Then every stratum of $\val(\salg)$ is semilinear.
\end{theorem}

Let us detail the notions used in this statement. If $\rcfield$ is a real closed field, then we say that a subset $\salg \subset \rcfield^{n}$ is \emph{basic semialgebraic} if it is of the form
\[
\salg = \{x \in \rcfield^{n} \colon \forall i = 1, \dots p, P_{i}(x) > 0  \land \forall i = p+1, \dots, q, P_{i}(x) = 0 \} \, ,
\]
where $P_{i} \in \rcfield[X_{1}, \dots, X_{n}]$ are polynomials. We say that $\salg$ is \emph{semialgebraic} if it is a finite union of basic semialgebraic sets. 
Similarly, if $\doag$ is a divisible ordered abelian group, then we say that a set $\slin \subset \doag^{n}$ is \emph{basic semilinear} if it is of the form
\[
\slin = \{g \in \doag^{n} \colon \forall i  = 1, \dots, p, f_{i}(g) > h^{(i)}, \forall i = p+1, \dots, q, f_{i}(g) = h^{(i)} \} \, ,
\]
where $f_{i} \in \Z[X_{1}, \dots, X_{n}]$ are homogeneous linear polynomials with integer coefficients and $h^{(i)} \in \doag$. We say that $\slin$ is \emph{semilinear} if it is a finite union of basic semilinear sets.

Finally, since we are interested in valuations of semialgebraic sets defined in valued fields, we work with $\doag \cup \{ \zero \}$ rather than $\doag$. Any set $S \subset (\doag \cup \{ \zero \})^{n}$ is naturally stratified as follows: the \emph{support} of a point $x \in (\doag \cup \{ \zero \})^n$ is defined as the set of indices $k \in [n]$ such that $x_k \neq \zero$. Given a nonempty subset $K \subset [n]$, and a set $S \subset (\doag \cup \{ \zero \})^n$, we define the \emph{stratum of $S$  associated with $K$} as the subset of $\doag^{\card K}$ formed by the projection $(x_k)_{k \in K}$ of the points $x \in S$ with support~$K$.

The rest of \cref{sec:model} is devoted to the presentation of the proof of \cref{theorem:image_finite_union_of_ri_polyh}, 
which relies on model theoretic results in valued fields. 
After a preliminary section on model theory (\cref{sec:languages}), we explain how \cref{theorem:image_finite_union_of_ri_polyh} is obtained from a quantifier elimination technique in valued fields of Denef and Pas (\cref{sec:qe}).

\subsection{Languages and structures}\label[section]{sec:languages}

In this section we recall some basic notions from model theory. We refer to \cite[Chapter~1]{marker_model_theory} and \cite[Chapter~1]{tent_ziegler_model_theory} for more information. In model theory, a \emph{language} $\lang$ is a collection of \emph{symbols} that are divided into three sets: a set of \emph{constant symbols}, a set of \emph{function symbols}, and a set of \emph{relation symbols}. 
For example, $\logroups \coloneqq (0, +, \le)$ is the language of ordered groups, while $\lorings \coloneqq (0, 1, +, -, \cdot, \le)$ is the language of ordered rings. 

An \emph{$\lang$-structure} is a tuple $\lstructure \coloneqq (\universe, \lang)$,
where $\universe$ is a nonempty set (called a \emph{domain}) and every symbol of $\lang$ can be interpreted in $\universe$. For instance, if $\oag = (\oag, 0, +, \le)$ is an ordered abelian group, then we can interpret the symbol $0$ as zero in $\oag$, the symbol $+$ as addition, and the symbol $\le$ as order in $\oag$. 
Thus, every ordered abelian group is an $\logroups$-structure.

The formalism introduced above 
enables us to study the first-order formulas over $\lang$ (or \emph{$\lang$-formulas}). The atoms of these formulas are constructed by applying relation symbols to terms built out of variables, and functions and constants from $\lang$.

Given an $\lang$-formula $\psi$ and a variable $x$, an occurrence of $x$ is said to be \emph{bound} if it is located within the scope of a subformula of the form $\forall x \dots$ or $\exists x \dots$. Other occurrences of the variable $x$ are said to be \emph{free}. By extension, the variable $x$ is said to be \emph{free} when it occurs freely in the formula $\psi$. Up to renaming some of the variables, we can suppose that free variables do not have bound occurrences.

We often denote an $\lang$-formula $\psi$ as $\psi(X)$, where $X = (x_{1}, \dots, x_{n})$ is a string of free variables that occur in $\psi$. If $\lstructure$ is an $\lang$-structure with domain $\universe$ and we fix a vector $\overbar{X} \in \universe^{n}$, then $\psi(\overbar{X})$ can be interpreted as a meaningful statement about $\lstructure$. This statement can be either true or false. For example, if we fix an ordered abelian group $\oag$, then the $\logroups$-formula $\forall x_{1} (x_{1} \ge 0 \limplies \exists x_{2} (x_{2} \ge 0 \land x_{1} = x_{2} + x_{2}))$ has no free variables. It is interpreted in $\oag$ as ``for every nonnegative element $x_{1} \in \oag$, there exist a nonnegative element $x_{2} \in \oag$ such that $x_{1}$ is equal to $x_{2}$ added to $x_{2}$.'' Note that this is true if we take $\oag = (\Q, +, \le)$, but false if we take $\oag = (\Z, +, \le)$. Similarly, the $\logroups$-formula $\exists x_{2} (x_{1} = x_{2} + x_{2})$ has one free variable $x_{1}$. If we take $\oag = (\Z, +, \le)$, then $\psi(2)$ is true, but $\psi(1)$ is false.
We denote ``$\psi(\overbar{X})$ is true in $\lstructure$'' as $\lstructure \models \psi(\overbar{X})$. A formula without free variables is called a \emph{sentence}. A set $\definable \subset \universe^{n}$ is called \emph{definable (in $\lang$)} if there exists a number $m \ge 0$, a vector $\overbar{b} \in \universe^{m}$, and an $\lang$-formula $\psi(x_{1}, \dots, x_{n+m})$ such that
\[
\definable = \{x \in \universe^{n} \colon \lstructure \models \psi(x_{1}, \dots, x_{n}, \overbar{b}) \} \, .
\]

\begin{example}\label[example]{ex:semilinear_from_formulas}
Take an $\logroups$-structure $\lstructure = (\doag, 0, +, \le)$, where $\doag$ is a divisible ordered abelian group. Suppose that $\psi(x_{1}, \dots, x_{n+m})$ is a quantifier-free $\logroups$-formula (i.e., a formula that does not contain quantifier symbols). Then $\slin = \{x \in \doag^{n} \colon \lstructure \models \psi(x_{1}, \dots, x_{n}, \overbar{b}) \}$ is a semilinear set. Conversely, every semilinear set can be written in such form.
\end{example}

If $\lang$ is a language, then any set of $\lang$-sentences is called a \emph{theory}. In our context, one can think that a theory is a set of axioms. If $\theory$ is a fixed theory in $\lang$, then we say that an $\lang$-structure $\lstructure$ is a \emph{model} of $\theory$ when we have $\lstructure \models \psi$ for every $\psi \in \theory$. 
Furthermore, if $\psi$ is an $\lang$-sentence that does not necessarily belong to $\theory$, then we say that $\psi$ is a \emph{logical consequence of $\theory$}, if $\psi$ is true in every model of $\theory$. We say that $\lang$-formulas $\psi(X)$, $\phi(X)$ are \emph{equivalent in $\theory$} if the sentence $\forall x_{1} \dots \forall x_{n} \ \psi(X) \liff \phi(X)$ is a logical consequence of $\theory$. We say that the theory $\theory$ admits \emph{quantifier elimination} if every $\lang$-formula 
is equivalent in $\theory$ to a quantifier-free formula. Finally, we say that a theory $\theory$ is \emph{complete} if for every $\lang$-sentence $\psi$, either $\psi$ or $\lnot \psi$ is a logical consequence of $\theory$.

\begin{example}
The theory of real closed fields, denoted $\thRcf$, is a theory in the language of ordered rings $\lorings$. It consists of the usual axioms of ordered fields, the axiom $\forall x_{1} (x_{1} \ge 0 \limplies \exists x_{2} (x_{1} = x_{2} \cdot x_{2}))$ that governs the existence of square roots, and an infinite set of axioms that states the fact that every polynomial of an odd degree has a root. In other words, for every $n \ge 1$, $\thRcf$ contains the axiom $\forall x_{0} \dots \forall x_{2n} \exists x (x^{2n+1} + x_{2n}x^{2n} + \dots + x_{1}x + x_{0} = 0)$. A classical result due to Tarski states that this theory admits quantifier elimination and is complete (see \cite[Theorem~3.3.15 and Corollary~3.3.16]{marker_model_theory}). As an immediate corollary one sees that if $\rcfield$ is a real closed field, then a set $\salg \subset \rcfield^{n}$ is definable in $\lorings$ if and only if it is semialgebraic.
\end{example}

In the next section, we use divisible ordered abelian groups which arise as value groups of nonarchimedean real closed fields. Since the valuation map may evaluate to $-\infty$, we need to deal with divisible ordered abelian groups with bottom element. In more details, we denote by $\logb \coloneqq (0, -\infty, +, \le)$ the language of ordered groups with bottom element. The theory of nontrivial divisible ordered abelian groups with bottom element, denoted $\thDoagb$, consists of the axioms of divisible ordered abelian groups, the nontriviality axiom $\exists y (y \neq 0 \land y \neq -\infty)$, and the axioms that extend the addition and order to $-\infty$, namely $\forall y ({-\infty} + y = -\infty)$ and $\forall y (y \ge -\infty)$. As stated in the next proposition, this theory admits quantifier elimination and is complete. It follows from the fact that the same result holds in the case of groups without bottom element \cite[Corollary 3.1.17]{marker_model_theory}. 

\begin{proposition}\label[proposition]{theorem:qe_doagb}%
The theory $\thDoagb$ admits quantifier elimination and is complete. Moreover, any $\logb$-formula $\theta(Y)$ with $Y = (y_{1}, \dots, y_{m})$ and $m \ge 1$ is equivalent to a quantifier-free formula of the form
\[
\biglor_{\Sigma \subset [m]} \bigl( (\forall \sigma \in \Sigma,  y_{\sigma} \neq -\infty \bigr) \ \land \ (\forall \sigma \notin \Sigma, y_{\sigma} = -\infty) \ \land \ \psi_{\Sigma} \bigr) \, ,
\]
where every $\psi_{\Sigma}$ is a quantifier-free $\logroups$-formula over a subset of variables in $\{y_\sigma\}_{\sigma \in \Sigma}$.
\end{proposition}
This proposition can be easily proven from \cite[Corollary~3.1.17]{marker_model_theory} using double induction over~$m$ and the length of $\theta$. We omit the proof for brevity. We emphasize that every $\psi_{\Sigma}$ is a $\logroups$-formula, i.e., a formula that does not contain the symbol $-\infty$. As a consequence of \cref{theorem:qe_doagb} and the discussion in \cref{ex:semilinear_from_formulas}, we get the following characterization of definable sets.
\begin{corollary}\label[corollary]{corollary:definable_semilinear}
Suppose that $\doag$ is a nontrivial divisible abelian group. Then $\slin \subset (\doag \cup \{-\infty \})^{n}$ is definable in $\logb$ if and only if every stratum of $\slin$ is semilinear.
\end{corollary}

\subsection{Quantifier elimination in real closed valued fields}\label[section]{sec:qe}
In this section, we want to show quantifier elimination over real closed fields equipped with a nontrivial and convex valuation. We suppose that $\vfield$ is a real closed field and $\val \colon \vfield \to \vgroup \cup \{-\infty\}$ is a valuation that is nontrivial and convex. 
We denote by $\resfield$ the residue field of $(\vfield, \vgroup, \val)$, and by $\angular$ we denote any angular component of this field. Under these conditions, $\vgroup$ is divisible and $\resfield$ is real closed, as noted in \cref{sec:valued_fields}.
In order to describe such structures, we consider
the following \emph{three-sorted language}
\[
\lvrcf \coloneqq (\lang_{\vfield}, \lang_{\vgroup}, \lang_{\resfield}, \val, \angular) \, .
\]
Here, $\lang_{\vfield}$ and $\lang_{\resfield}$ denote the language of ordered rings (respectively associated with $\vfield$ and~$\resfield$), $\lang_{\vgroup}$ denotes the language of ordered groups with bottom element, $\val$ is a symbol for valuation map, and $\angular$ is a symbol for angular component. In the language $\lvrcf$, any formula has 
three kinds of variables, one kind for every sort. Let $x_{1}, x_{2}, \dots$ denote the variables associated with $\vfield$, $y_{1}, y_{2}, \dots$ denote the variables associated with $\vgroup \cup \{-\infty \}$, and $z_{1}, z_{2}, \dots$ denote the variables associated with $\resfield$. If $\theta$ is a $\lvrcf$-formula, then we denote it as $\theta(X, Y ,Z)$, where $X, Y, Z$ are sequences of free variables associated with $\vfield$, $\vgroup \cup \{-\infty \}$, $\resfield$ respectively. The constant, function, and relation symbols of the language $\lvrcf$ are implicitly typed. For instance, the addition symbol of $\lang_\vfield$ takes two elements of the sort $\vfield$, and returns an element of the same sort. The symbol $\val$ yields an element of the sort $\vgroup$ from an element of the sort~$\vfield$. Then, $\lvrcf$-formulas are built from the symbols of the language $\lvrcf$ and variables in such a way that every term and atom is well typed. We refer to~\cite[Chapter~1]{tent_ziegler_model_theory} for a formal treatment of multisorted languages.

Let us denote by $\thVrcf$ the theory of valued fields with angular component which are real closed and have a nontrivial and convex valuation. More precisely, the theory consists of the axioms of the theory of real closed fields for $\vfield$, the axioms of the theory of ordered abelian groups with bottom element for $\vgroup \cup \{-\infty\}$, the axioms of ordered fields for $\resfield$, the axioms specifying that $\val$ is a nontrivial and convex valuation, and the axioms specifying that $\angular$ is an angular component. Note that the latter axioms imply that $\resfield$ is indeed the residue field of $(\vfield, \vgroup, \val)$. In the next theorem, we show that this theory admits quantifier elimination. The cornerstone of the proof is a result due to Pas~\cite{pas_cell_decomposition,pas_mixed_characteristic}, which establishes that the theory of henselian valued fields with angular component admits elimination of quantifiers over the $\vfield$-variables. We refer to \cite{cluckers_analytic_cell_decomposition} for more recent generalizations of Pas's result.

\begin{theorem}\label[theorem]{theorem:qe_vrcf}%
The theory $\thVrcf$ admits quantifier elimination and is complete. Moreover, any $\lvrcf$-formula $\theta(X, Y, Z)$ is equivalent in $\thVrcf$ to a formula of the form
\begin{equation*}
\begin{aligned}
\biglor_{i = 1}^{m} \Bigl( \phi_{i}\bigl(\val(f_{i1}(X)), \dots, \val(f_{ik_{i}}(X)),Y \bigr) \land \psi_{i}\bigl(\angular(f_{i(k_{i}+1)}(X)), \dots, \angular(f_{il_{i}}(X)), Z \bigr) \Bigr) \, ,
\end{aligned}
\end{equation*}
where, for every $i = 1, \dots, m$, 
$f_{i1}, \dots, f_{il_{i}} \in \Z[X]$ are polynomials with integer coefficients, $\phi_{i}$ is a quantifier-free $\lang_{\vgroup}$-formula, and $\psi_{i}$ is a quantifier-free $\lang_{\resfield}$-formula.
\end{theorem}
\begin{proof}
Let $\theta(X, Y, Z)$ denote any $\lvrcf$-formula. Recall that the order $\le$ in any real closed field can be defined as $x_{1} \le x_{2} \iff \exists x_{3} ( x_{2} - x_{1} = x_{3}^{2})$. This enables us to inductively eliminate all occurrences of the symbols $\le$ of the languages $\lang_\vfield$ and $\lang_\resfield$. Therefore,  $\theta(X, Y, Z)$ is equivalent in $\thVrcf$ to a formula $\hat{\theta}(X, Y, Z)$ without the symbol~$\le$. Moreover, by \cite[Theorem~4.3.7]{engler_prestel_valued_fields}, $(\vfield, \vgroup, \val)$ is henselian. 
This enables us to apply the quantifier elimination of Pas \cite[Theorem~4.1]{pas_cell_decomposition}. (To be more precise, we use the formulation given in \cite[Theorem~4.2]{cluckers_analytic_cell_decomposition}.) As a result, $\hat{\theta}(X, Y, Z)$ is equivalent in $\thVrcf$ to a formula of the form
\begin{equation}\label[equation]{eq:qe_vrcf_proof}
\begin{aligned}
\biglor_{i = 1}^{m} \Bigl( \phi_{i}\bigl(\val(f_{i1}(X)), \dots, \val(f_{ik_{i}}(X)),Y \bigr) \land \psi_{i}\bigl(\angular(f_{i(k_{i}+1)}(X)), \dots, \angular(f_{il_{i}}(X)), Z \bigr) \Bigr) \, ,
\end{aligned}
\end{equation}
where, for every $i = 1, \dots, m$, 
$f_{i1}, \dots, f_{il_{i}} \in \Z[X]$ are polynomials with integer coefficients, $\phi_{i}$ is an $\lang_{\vgroup}$-formula, and $\psi_{i}$ is an $\lang_{\resfield}$-formula. Then, we apply \cref{theorem:qe_doagb} and \cite[Theorem~3.3.15]{marker_model_theory} to eliminate the quantifiers in the formulas $\phi_{i}$ and $\psi_{i}$. This shows the last part of the statement.

In the case where $\theta$ is a sentence, the formulas $\phi_i$ and $\psi_i$ are also sentences. The completeness results in \cref{theorem:qe_doagb} and~\cite[Corollary~3.3.16]{marker_model_theory} applied to each subformula $\phi_i$ and $\psi_i$ in~\cref{eq:qe_vrcf_proof} allow to prove that either $\theta$ or $\lnot \theta$ is a logical consequence of $\thVrcf$.
\end{proof}
As a corollary, we obtain \cref{theorem:image_finite_union_of_ri_polyh}.
\begin{proof}[Proof of~\cref{theorem:image_finite_union_of_ri_polyh}]
Let $\vgroup$ denote the value group of $\vfield$ and $\resfield$ denote the residue field. The structure $\lstructure = (\vfield, \vgroup \cup \{-\infty\}, \resfield, \lvrcf)$ is a model of $\thVrcf$. Let $\phi(x_{1}, \dots, x_{n+m})$ be an $\lang_{\vfield}$-formula and $\overbar{b} \in \vfield^{m}$ be a vector such that $\salg = \{x \in \vfield^{n} \colon \vfield \models \phi(x, \overbar{b}) \}$. Take the formula $\theta(x_{n+1}, \dots, x_{n+m}, y_{1}, \dots, y_{n})$ in $\lvrcf$ defined as
\[
\exists x_{1} \dots \exists x_{n} \Bigl(  \phi(x_{1}, \dots, x_{n+m})  \land \val(x_{1}) = y_{1} \land \dots \land \val(x_{n}) = y_{n} \Bigr) \, .
\]
We obviously have 
\[
\val(\salg) = \{ y \in (\Gamma \cup \{ \zero\})^{n} \colon \lstructure \models \theta(\overbar{b}, y)\} \, .
\]
By~\cref{theorem:qe_vrcf}, $\theta$ is equivalent to a formula of the form
\begin{equation*}
\begin{aligned}
\biglor_{i = 1}^{m} \Bigl( \phi_{i}\bigl(\val(f_{i1}(X)), \dots, \val(f_{ik_{i}}(X)),Y \bigr) \land \psi_{i}\bigl(\angular(f_{i(k_{1}+1)}(X)), \dots, \angular(f_{il_{i}}(X)) \bigr) \Bigr) \, ,
\end{aligned}
\end{equation*}
where we denote $X \coloneqq (x_{n+1}, \dots, x_{n+m})$, $Y \coloneqq (y_{1}, \dots, y_{n})$, every $\phi_{i}$ is an $\lang_\vgroup$-formula, every $\psi_{i}$ is an $\lang_\resfield$-formula, and $f_{i1}, \dots, f_{il_{i}}$ are polynomials with integer coefficients. If we fix $X$ to be equal to $\overbar{b}$, then this formula is equivalent to a formula of the form
\[
\biglor_{i \in I} \phi_i\bigl( \xi_{i1}, \dots, \xi_{ik_i},Y \bigr) \, ,
\]
where $I$ is a subset of $[m]$ and we denote $\val \bigl(f_{ik}(\overbar{b}) \bigr) = \xi_{ik} \in \vgroup$. Hence, $\val(\salg)$ is definable in $\lang_{\vgroup}$. By \cref{corollary:definable_semilinear}, $\val(\salg)$ has semilinear strata.
\end{proof}

\section{Closedness of tropical semialgebraic sets}\label[section]{sec:images}
In this section we strengthen \cref{theorem:image_finite_union_of_ri_polyh} by showing that the strata of $\val(\salg)$ are not only semilinear but also closed. More precisely, we show the following theorem. 
\begin{theorem}\label[theorem]{theorem:images_are_closed}
Let $\vfield$ be a real closed field equipped with a nontrivial and convex valuation~$\val$. Suppose that set $\salg \subset \vfield^{n}$ is semialgebraic. Then every stratum of $\val(\salg)$ is closed in the product topology of the order topology of the value group $\vgroup$. 

Furthermore, if $\salg$ is closed, then $\val(\salg)$ is closed in the product topology of the order topology of $\vgroup \cup \{-\infty \}$.
\end{theorem}

\begin{remark}\label[remark]{remark:alessandrini}
  \Cref{theorem:image_finite_union_of_ri_polyh,theorem:images_are_closed} should be compared with Theorem~3.11 and Theorem~4.10 of Alessandrini~\cite{alessandrini},
  dealing with log-limits of sets that are definable in certain o-minimal
  polynomially bounded structures.
  The fact that we obtain
  the present results
   using quantifier elimination has two advantages. First, it gives a constructive proof, while the proofs in~\cite{alessandrini} rely on compactness arguments. Second, our methods apply to arbitrary real closed valued fields, while the analysis of \cite{alessandrini} is restricted to fields
  whose value group is a subgroup of $\R$.
  On the other hand, the results of \cite{alessandrini} allow to go beyond the semialgebraic setting; they apply for instance to the o-minimal structure of real numbers with restricted analytic functions.
\end{remark}

\begin{remark}
We point out that every subset of $(\vgroup\cup\{-\infty\})^n$ that is closed in the product topology of the order topology of $\vgroup \cup \{-\infty \}$ has closed strata, but the converse is not true. For example, the set $\R^{2} \cup \{-\infty \}$ has closed strata, but is not closed in the product topology of the order topology of $\trop^{2}$. 
\end{remark}

To prove \cref{theorem:images_are_closed}, we first consider the case of Puiseux series, $\vfield = \puiseux$. 
The proof needs a few auxiliary lemmas. Hereafter, $\pospuiseux^{n} \coloneqq \{\bo x \in \puiseux \colon \forall k, \bo x_{k} > 0 \}$ denotes the open positive orthant of $\puiseux^{n}$. Let us fix a basic semialgebraic set $\bsalg \subset \pospuiseux^{n}$ defined as
\begin{equation}
\bsalg \coloneqq \{\bo x \in \pospuiseux^{n} \colon \forall i = 1, \dots p, \bo P_{i}(\bo x) > 0  \land \forall i = p+1, \dots, q, \bo Q_{i}(\bo x) = 0 \} \, \label[equation]{eq:basic_semialgebraic}
\end{equation}
for some polynomials $\bo P_{1}, \dots, \bo P_{p}, \bo Q_{p+1}, \dots, \bo Q_{q} \in \puiseux[X_{1}, \dots, X_{n}]$. Equivalently, we put $\bsalg$ under the form
\begin{equation}
\bsalg = \{\bo x \in \pospuiseux^{n} \colon \forall i = 1, \dots p, \bo P_{i}(\bo x) > 0  \land \forall i = p+1, \dots, q, \bo P_{i}(\bo x) \ge 0 \} \, , 
\label[equation]{eq:basic_semialgebraic_ineq}
\end{equation}
where we set $\bo P_{i} \coloneqq - \bo Q^{2}_{i}$ for all $i = p+1, \dots, q$. Denote $P_{i} \coloneqq \tropPol(\bo P_{i})$ for all $i = 1, \dots, q$. In the next lemma, we highlight a property of the full-dimensional cells of the complex $\complex(P_{1}, \dots, P_{q})$ whose interior is contained in $\val(\bsalg)$.
\begin{lemma}\label[lemma]{lemma:formal_partition_of_space}
Suppose that $\polyh$ is a full-dimensional cell of $\complex(P_{1}, \dots, P_{q})$ such that $\inter(\polyh) \cap \val(\bsalg) \neq \emptyset$. Let $w \in \inter(\polyh)$, and $\bo w \in \val^{-1}(w)\cap \pospuiseux^{n}$ be an arbitrary lift. Then $\bo w \in \bsalg$.
\end{lemma}

\begin{proof}
Take a point $\bo z \in \bsalg$ such that $z \coloneqq \val(\bo z) \in \inter(\polyh)$. For every $i = 1, \dots, q$ we have $\bo P_{i}(\bo z) \ge 0$. By \cref{lemma:homomorphism} and the fact $\val$ is order preserving, we obtain $P_{i}^{+}(z) \ge P_{i}^{-}(z)$. Since $\polyh$ is a full-dimensional cell of $\complex(P_{1}, \dots, P_{q})$, we have the equality $\inter(\polyh) = \cap_{i = 1}^{q} \inter(\polyh_{i})$, where, for every $i$, $\polyh_{i}$ is a full-dimensional cell of $\complex(P_{i})$. In particular, $\domin(P_{i}, z)$ has only one element and we have $P_{i}^{+}(z) > P_{i}^{-}(z)$. Furthermore, we have $\domin(P_{i}, z) = \domin(P_{i}, w)$ for any point $w \in \inter(\polyh)$. This implies that $P_{i}^{+}(w) > P_{i}^{-}(w)$. Therefore, if $\bo w \in \val^{-1}(w) \cap \pospuiseux^{n}$ is an arbitrary lift of $w$, then by \cref{lemma:homomorphism} we have $\val(\bo P_{i}^{+}(\bo w)) > \val(\bo P_{i}^{-}(\bo w))$ and hence $\bo w \in \bsalg$.
\end{proof}

\begin{lemma}\label[lemma]{lemma:linear_transformation}
Let $A \in \Q^{m \times n}$ be any matrix. Define a function $f \colon \pospuiseux^{n} \to \pospuiseux^{m}$ as
\[
f(\bo x)_{i} \coloneqq \bo x_{1}^{A_{i1}} \bo x_{2}^{A_{i2}} \dots \bo x_{n}^{A_{in}} \, .
\]
Let $\bsalg \subset \pospuiseux^{n}$ be any semialgebraic set. Then $f(\bsalg) \subset \pospuiseux^{m}$ is semialgebraic and we have $\val(f(\bsalg)) = A\bigl(\val(\bsalg)\bigr)$.
\end{lemma}
\begin{proof}
The first claim follows from the fact that the class of semialgebraic sets is closed under semialgebraic transformations \cite[Proposition~2.83]{basu_pollack_roy_algorithms}. The second claim follows from the identity $\val(f(\bo x)_{i}) = A_{i}\bigl( \val(\bo x) \bigr)$.
\end{proof}
\begin{lemma}\label[lemma]{lemma:images_are_closed_one_stratum}
Suppose that $\bsalg \subset \pospuiseux^{n}$ is a semialgebraic set. Then $\val(\bsalg) \subset \R^{n}$ is a union of finitely many closed polyhedra.
\end{lemma}
\begin{proof}
We proceed by induction over the dimension $n$.
First, suppose that $n = 1$. Since $\puiseux$ is a real closed field, 
every semialgebraic set in $\puiseux$ is a finite union of points and open intervals.
Observe that the image by the valuation of an open interval 
in $\pospuiseux$ is an interval that is closed in $\R$. 
Therefore, the claim is true for $n = 1$.

Second, suppose that the claim holds in dimension $n - 1$. Observe that it is enough to prove the claim for basic semialgebraic sets. Fix a basic semialgebraic set $\bsalg \subset \pospuiseux^{n}$ as in \cref{eq:basic_semialgebraic_ineq} and take the polyhedral complex $\complex \coloneqq \complex(P_{1}, \dots, P_{q})$. Let $\tilde{\polyh}_{1}, \dots, \tilde{\polyh}_{r}$ denote the cells of $\complex$. 
By \cref{theorem:image_finite_union_of_ri_polyh}, $\val(\bsalg)$ is a finite union of relatively open polyhedra. Denote these polyhedra by $\ri(\tilde{\face}_{1}), \dots, \ri(\tilde{\face}_{s})$, where each $\tilde{\face_{j}}$ is a closed polyhedron and $\ri$ denotes the relative interior. 
For every $(i,j)$, let $\polyh_{ij}$ be a polyhedron such that
\[
\ri(\polyh_{ij}) = \ri(\tilde{\polyh}_{i}) \cap \ri(\tilde{\face}_{j}) \, .
\]
Observe that $\val(\bsalg)$ is a union of $\ri(\polyh_{ij})$. 
We consider an element $w^*$ of $ \cl(\val(\bsalg))$. Let us look at two cases.

\begin{asparaenum}[{Case} I:]
\item There is a full-dimensional polyhedron $\polyh_{ij}$ such that $w^{*} \in \polyh_{ij}$. In this case, let $\hplane = \{w \in \R^{n} \colon \langle a, w \rangle = \langle a, w^* \rangle \}$ be any hyperplane intersecting the interior of $\polyh_{ij}$, and such that $a \in \Q^n$. Consider $w^{(1)}, w^{(2)}, \dots$ a sequence such that $w^{(h)} \in \hplane \cap \inter(\polyh_{ij})$ for all $h$ and $w^{(h)} \to w^{*}$. Take the set $\bo Y \subset \pospuiseux^{n}$ defined as
\[
\bo Y = \bsalg \cap \Bigl\{\bo x \in \pospuiseux^{n} \colon \prod_{k \in [n]} \bo x_{k}^{a_{k}} = t^{\langle a, w^* \rangle} \Bigr\} \, .
\]
For every $h$ define $\bo w^{(h)} \in \val^{-1}(w^{(h)}) \cap \pospuiseux^{n}$ as $\bo w^{(h)}_{k} = t^{w^{(h)}_{k}}$ for all $k \in [n]$. Note that every $w^{(h)}$ belongs to the interior of the full-dimensional polyhedron $\tilde{\polyh}_i$. Consequently, $\bo w^{(h)}$ belongs to $\bo Y$ by \cref{lemma:formal_partition_of_space}. Take $l \in [n]$ such that $a_{l} \neq 0$ and let $\bo \pi \colon \pospuiseux^{n} \to \pospuiseux^{n-1}$ denote the projection that forgets the $l$-th coordinate. 
Similarly, let $\pi \colon \R^{n} \to \R^{n-1}$ denote the projection that forgets the $l$-th coordinate. By the induction hypothesis and \cref{lemma:linear_transformation}, $\val(\bo \pi(\bo Y))$ is a closed subset of $\R^{n-1}$ and we have $\val(\bo \pi (\bo Y)) = \pi(\val(\bo Y))$. The sequence $\pi(w^{(h)})$ converges to $\pi(w^{*})$. Therefore, we have $\pi(w^{*}) \in \pi(\val(\bo Y))$. In other words, there exists a point $\bo w^{*} \in \bo Y$ such that $\pi(\val(\bo w^{*})) = \pi(w^{*})$. Moreover, we have $\val(\bo w^{*}) \in \hplane$ and $w^{*} \in \hplane$. Since $a_{l} \neq 0$, this implies that $\val(\bo w^{*}) = w^{*}$. Therefore $w^{*} \in \val(\bspectra)$.

\item If $w^{*}$ does not belong to any full-dimensional polyhedron $\polyh_{ij}$, then we denote by $I$ the set of all indices $(i,j)$ such that $\polyh_{ij}$ contains $w^{*}$. We can take $\rho > 0$ so small that the closed Chebyshev ball $B(w^{*}, \rho)$ does not intersect any polyhedron $\polyh_{ij}$ with $(i,j) \notin I$. Let $w^{(1)}, w^{(2)}, \dots$ be a convergent sequence of elements of $\R^{n}$, $w^{(h)} \to w^{*}$ such that $w^{(h)} \in \val(\bsalg)$ for all $h$. 
Every polyhedron $\polyh_{ij}$ such that $(i,j) \in I$ is not full-dimensional. Therefore, it is included in an affine hyperplane $\hplane_{ij}$. Let $X = \bigcup_{(i,j) \in I} \hplane_{ij}$ be a union of these hyperplanes. Observe that we have $w^{*} \in X$ and that $\val(\bsalg) \cap B(w^{*}, \rho) \subset X$. Let $v \in \Q^{n}$ be any rational vector such that $v \notin (X - w^{*})$. (Here, by $X- w^{*}$ we mean the translation of $X$ by vector $-w^{*}$.) Note that the affine line $w^{*} + \vecspan(v)$ intersects $X$ only in $w^{*}$. 

Let $A \in \Q^{(n-1) \times n}$ be a rational matrix such that $\ker(A) = \vecspan(v)$. Take the function $f \colon \pospuiseux^{n} \to \pospuiseux^{n-1}$ defined as
\[
(f(\bo x))_{k} \coloneqq \prod_{l = 1}^{n} \bo x_{l}^{A_{kl}} \, , \ k = 1,2,\dots, n-1 \, .
\]
Let $\bo U \coloneqq \{\bo x \in \pospuiseux^{n} \colon \forall l, \bo x_{l} \in [t^{w^{*}_{l} - \rho}, t^{w^{*}_{l} + \rho}] \}$. By \cref{lemma:linear_transformation}, the set $f(\bsalg \cap \bo U) \subset \pospuiseux^{n-1}$ is semialgebraic and we have $\val(f(\bsalg \cap \bo U)) = A(\val(\bsalg \cap \bo U))$. Therefore, by the induction hypothesis, the set $A(\val(\bsalg \cap \bo U))$ is closed.  For every $w^{(h)}$, let $\bo w^{(h)} \in \bsalg$ denote 
any element of $\bsalg$ such that $\val(\bo w^{(h)}) = w^{(h)}$. For $h$ large enough we have $w^{(h)} \in B(w^{*}, \rho/2)$ and hence $\bo w^{(h)} \in \bsalg \cap \bo U$. Moreover, the sequence $Aw^{(h)}$ converges to $Aw^{*}$. Since $A(\val(\bsalg \cap \bo U))$ is closed, there is $\bo w^{*} \in \bsalg \cap \bo U$ such that $Aw^{*} = A \val(\bo w^{*})$. As $\bo w^{*} \in \bo U$, we have $\val(\bo w^{*}) \in B(w^{*}, \rho)$. Therefore 
\[
\val(\bo w^{*}) \in  \Bigl( (w^{*} + \vecspan(v)) \cap B(w^{*}, \rho) \cap \val(\bsalg) \Bigr)\, .
\]
On the other hand, we have $\val(\bsalg) \cap B(w^{*}, \rho) \subset X$ and $(w^{*} + \vecspan(v)) \cap X = w^{*}$. Hence $\val(\bo w^{*}) = w^{*}$ and $w^{*} \in \val(\bsalg)$.
\end{asparaenum}
 \end{proof}
\begin{lemma}\label[lemma]{point_with_zeros}
Suppose that $\bsalg \subset \puiseux^{n}$ is a nonempty bounded closed semialgebraic set. Let $K \subset [n]$ be a set of indices such that for every $\bo a \in \pospuiseux$ the set $\{\bo x \in \bsalg \colon \forall k \in K, \,  \bo x_{k} \in [0, \bo a] \}$ is nonempty. Then there exists a point $\bo y \in \bsalg$ such that $\bo y_{k} = 0$ for all $k \in K$.
 \end{lemma}

\begin{proof}
We prove that the statement holds for any real closed field $\rcfield$. Fix an $\lorings$-formula $\psi(x_{1}, \dots, x_{n+m})$. For every vector $\overbar{b} \in \rcfield^{m}$ we define the semialgebraic set $\salg_{\overbar{b}}$ by
\[
\salg_{\overbar{b}} \coloneqq \{x \in \rcfield^{n} \colon \rcfield \models \psi(x_{1}, \dots, x_{n}, \overbar{b}) \} \, .
\]
The statement ``for all $(x_{n+1}, \dots, x_{n+m})$, if the set $\salg_{(x_{n+1}, \dots, x_{n+m})}$ is nonempty, bounded, closed, and the set $\{x \in \salg_{(x_{n+1}, \dots, x_{n+m})} \colon \forall k \in K, \,  x_{k} \in [0, a] \}$ is nonempty for every $a > 0$, then there exists a point $y \in \salg_{(x_{n+1}, \dots, x_{n+m})}$ such that $y_{k} = 0$ for all $k \in K$'' is a sentence in the language of ordered rings $\lorings$. It is true in $\R$, hence it is true in $\rcfield$ by the model completeness of real closed fields (\cite[Corollary~3.3.16]{marker_model_theory}).
 \end{proof}
The lemmas above lead to the main theorem of this section.

\begin{proof}[Proof of~\cref{theorem:images_are_closed}]
We first prove the result for a semialgebraic set $\bsalg$ included in the closed positive orthant $\nnpuiseux^{n}$. Let $K \subset [n]$ be any nonempty subset and let $\bo X_{K} \subset \puiseux^{n}$ be the set defined as
\[
\bo X_{K} \coloneqq \{ \bo x \in \puiseux^{n} \colon \bo x_{k} \neq 0 \iff k \in K \}.
\]
The sets $\bo X_{K}$ and subsequently $\bsalg \cap \bo X_{K}$ are semialgebraic. Let $\bo \pi \colon \puiseux^{n} \to \puiseux^{K}$ denote the projection on the coordinates from $K$. Similarly, let $\pi \colon \trop^{n} \to \trop^{K}$ denote the projection on the coordinates from $K$. Observe that the stratum of $\val(\bsalg)$ associated with $K$ is equal to $\pi(\val(\bsalg \cap \bo X_{K})) = \val(\bo \pi(\bsalg \cap \bo X_{K}))$. Moreover, the set $\bo \pi(\bsalg \cap \bo X_{K})$ is included in $\pospuiseux^{K}$. Therefore the first claim follows from \cref{lemma:images_are_closed_one_stratum}. 

To prove the second claim, suppose that $\bsalg$ is closed. Let $x \in \trop^{n}$ be any point that does not belong to $\val(\bsalg)$. Let $K \subset [n]$ be the support of $x$. For any $M, N > 0$ we denote $I_{k}(M,N) = [-\infty, -M[$ if $k \notin K$ and $I_{k}(M,N) ={]x_{k} - \frac{1}{N}, x_{k} + \frac{1}{N}[}$ otherwise. Similarly, we denote $\bo I_{k}(M,N) = [0, t^{-M + 1}] \subset \nnpuiseux$ for $k \notin K$ and $\bo I_{k}(M,N) = [t^{x_{k} - \frac{2}{N}}, t^{x_{k} + \frac{2}{N}}] \subset \pospuiseux$ otherwise. We want to show that there is an open neighborhood of $x$ that does not belong to $\val(\bsalg)$. Suppose that this is not the case. Then, for any $M,N > 0$, the set $\prod_{k = 1}^{n} I_{k}(M,N)$ contains a point from $\val(\bsalg)$. Therefore, the set $\bsalg^{(M, N)} \coloneqq \bsalg \cap \prod_{k = 1}^n \bo I_k(M,N)$ is nonempty, on top of being closed and bounded. If we fix $N > 0$, then, by \cref{point_with_zeros}, there is a point $\bo y^{(N)} \in \bsalg^{(1,N)}$ such that $\bo y^{(N)}_{k} = 0$ for all $k \notin K$. In other words, the set $\bo \pi(\bsalg \cap \bo X_{K})$ contains a point that belongs to $\prod_{k \in K} [t^{x_{k} - \frac{2}{N}}, t^{x_{k} + \frac{2}{N}}]$. Hence, the stratum of $\val(\bsalg)$ associated with $K$ contains a point that belongs to $\prod_{k \in K} [x_{k} - \frac{2}{N}, x_{k} + \frac{2}{N}]$. Since this is true for all $N > 0$, and the strata of $\val(\bsalg)$ are closed, we have $x \in \val(\bsalg)$, which gives a contradiction.
 
Second, suppose that $\bsalg \subset \puiseux^{n}$ is any semialgebraic set. Given a vector $\delta \in \{+1, -1\}^{n}$, we denote by $f_\delta$ the involution which maps $\bo x \in \puiseux^n$ to the vector with entries $\delta_k \bo x_k$. With this notation, $\bsalg$ is the union of the sets of the form $\bsalg \cap f_{\delta}(\nnpuiseux^n)$. Moreover, the set $f_\delta(\bsalg \cap f_\delta(\nnpuiseux^n))$ is a semialgebraic set included in $\nnpuiseux^{n}$, and its image under the valuation map coincides with that of $\bsalg \cap f_\delta(\nnpuiseux^n)$. The claim follows by applying the result of the previous paragraph to each of the sets $f_\delta(\bsalg \cap f_\delta(\nnpuiseux^n))$.

To prove the claim for an arbitrary field $\vfield$ we use \cref{theorem:qe_vrcf}. We fix an $\lorings$-formula $\psi(x_{1}, \dots, x_{n+m})$. For every vector $\overbar{b} \in \vfield^{m}$ we can look at the semialgebraic set
\[
\salg_{\overbar{b}} \coloneqq \{x \in \vfield^{n} \colon \vfield \models \psi(x_{1}, \dots, x_{n}, \overbar{b}) \} \, .
\]
The statement ``for all $(x_{n+1}, \dots, x_{n+m})$, the image by valuation of the set $\salg_{(x_{n+1}, \dots x_{n+m})}$ has closed strata'' is a sentence in $\lvrcf$. It is true in $\puiseux$ and hence, by  the completeness result of \cref{theorem:qe_vrcf}, it is also true in $\vfield$. The same is true for the statement ``for all $(x_{n+1}, \dots, x_{n+m})$, if $\salg_{(x_{n+1}, \dots x_{n+m})}$ is closed, then its image by valuation is closed.''
\end{proof}

As a by-product, we get the following result, which generalizes the proposition of Develin and Yu \cite[Proposition~2.9]{develin_yu} on polyhedra to basic semialgebraic sets.
\begin{corollary}\label[corollary]{corollary:purity_commutation}
Suppose that $\bsalg \subset \pospuiseux^{n}$ is a semialgebraic set defined as
\[
\bsalg \coloneqq \{\bo x \in \pospuiseux^{n} \colon \bo P_{1}(\bo x) \mathrel{\Box}_{1} 0, \dots,  \bo P_{m}(\bo x) \mathrel{\Box}_{m} 0\} \, ,
\]
where $\bo P_{i} \in \puiseux[X_{1},\dots, X_{n}]$ are nonzero polynomials and $\Box \in \{\ge, > \}^{m}$. Let $P_{i} \coloneqq \tropPol(\bo P_{i})$ for all $i$ and suppose that $\poscomplex(P_{1}, \dots, P_{m})$ has regular support. Then
\[
\val(\bsalg) = \{x \in \R^{n} \colon \forall i, P_{i}^{+}(x) \ge P_{i}^{-}(x) \} \, .
\]
\end{corollary}
\begin{proof}
Denote $\positivesupport \coloneqq \{x \in \R^{n} \colon \forall i, P_{i}^{+}(x) \ge P_{i}^{-}(x) \}$ and suppose that $\bo x \in \bsalg$. Since $\val$ is order preserving, we have $\val(\bo x) \in \positivesupport$ by \cref{lemma:homomorphism}. Therefore $\val(\bsalg) \subset \positivesupport$. On the other hand, if we take any point $x$ such that $P_{i}^{+}(x) > P_{i}^{-}(x)$ for all $i$, then any lift $\bo x \in \val^{-1}(x) \cap \pospuiseux^{n}$ belongs to $\bsalg$. Hence, we have the inclusion
\[
\{x \in \R^{n} \colon \forall i, P_{i}^{+}(x) > P_{i}^{-}(x) \} \subset \val(\bsalg) \subset \{x \in \R^{n} \colon \forall i, P_{i}^{+}(x) \ge P_{i}^{-}(x) \} \,
\]
and the claim follows from \cref{lemma:purity_closes_strict,theorem:images_are_closed}.
\end{proof}

\begin{figure}[t]
\centering
\begin{tikzpicture}
\begin{scope}[scale = 0.6]
      \draw[gray!60, ultra thin] (0,0) grid (8,8);
      \draw[gray!80!black, thin, ->] (0,4) -- (8,4) node[below right] {$x_1$};
      \draw[gray!90!black, thin, ->] (4,0) -- (4,8) node[below left] {$x_2$};
      \fill[fill = lightgray, fill opacity = 0.7] (0,0) -- (0,3) -- (3,3) -- (3,0) -- cycle;
      \fill[fill = lightgray, fill opacity = 0.7] (8,5) -- (5,5) -- (5,8) -- (8,8) -- cycle;
        \draw[very thick] (0,3) -- (3,3) -- (3,0);
        \draw[very thick] (8,5) -- (5,5) -- (5,8);
        \draw[very thick] (3,3) -- (5,5); 
        \fill[black] (3,3) circle (0.15);
        \fill[black] (5,5) circle (0.15);
\end{scope}
\end{tikzpicture}
\caption{Polyhedral complex $\poscomplex(P)$.}\label[figure]{fig:newton_polytope_with_impure_complex}
\end{figure}

\begin{example}\label[example]{ex:hypersurface}
Take $P = 0 \tplus (X_{1}^{\tdot 2} \tdot X_{2}^{\tdot 2}) \tplus (2 \tdot X_{1} \tdot X_{2}) \tplus (\tminus 2 \tdot X_{1}^{\tdot 2}) \tplus (\tminus 2 \tdot X_{2}^{\tdot 2})$. Then $\poscomplex(P)$ is depicted in \cref{fig:newton_polytope_with_impure_complex}. This support of this complex is not regular and \cref{corollary:purity_commutation} does not apply. Indeed, take $\bo P(\bo x_{1}, \bo x_{2}) = 1 + \bo x_{1}^{2}\bo x_{2}^{2} + t^{2} \bo x_{1} \bo x_{2} - t^{2} \bo x_{1}^{2} - t^{2} \bo x_{2}^{2}$. We have $\tropPol(\bo P) = P$, but the set $\val(\{(\bo x_{1}, \bo x_{2}) \in \pospuiseux^{2} \colon \bo P(\bo x_{1}, \bo x_{2}) \ge 0 \})$ does not contain the open segment $](-1,-1), (1,1)[$.
\end{example}

\begin{remark}\label[remark]{remark:kapranov}
  The result of \cref{corollary:purity_commutation} may be thought of as a semialgebraic analogue of the Kapranov's theorem~\cite[Theorem~2.1.1]{einsiedler_kapranov}. This theorem characterizes an image by valuation of a complex hypersurface. A difference is that Kapranov's theorem involves only one polynomial equality constraint,
  while \cref{corollary:purity_commutation} applies
  to several
  polynomial inequality constraints. On the other hand, Kapranov's theorem does not require any additional assumptions, while the purity condition of \cref{corollary:purity_commutation} cannot be dropped, as shown by \cref{ex:hypersurface}. Furthermore, we note that Kapranov's theorem can be generalized
  to fields with higher rank value groups~\cite{aroca_hypersurfaces,banerjee_higher_rank}. We do not pursue this direction here, but we note that higher rank analogues of \cref{corollary:purity_commutation} can be deduced
  from it using the quantifier elimination result of \cref{theorem:qe_vrcf}.
\end{remark}

After this manuscript was submitted, the tropicalization of semialgebraic sets was further studied in the work of Jell, Scheiderer, and Yu~\cite{jell_scheiderer_yu}. The main focus of \cite{jell_scheiderer_yu} is to study the real analytification of semialgebraic sets, but the authors also obtain results about the topology of tropicalized sets, similar to the results proved above. In order to facilitate the comparison between the two papers, we now briefly discuss these results.

The authors of \cite{jell_scheiderer_yu} use a more expressive
notion of tropicalization that also captures the sign structure of the set. More precisely, instead of using the valuation map $\val \colon \puiseux \to \trop$, they use the signed valuation map $\sval \colon \puiseux \to \strop$. In order to equip the space of signed tropical numbers with a topology, we extend the order on $\trop$ to an order on $\strop$ in the natural way, by supposing that negative tropical numbers are smaller than the positive tropical numbers (so that $\tminus 3 < \tminus (-2) < \zero < (-2) < 3$). This induces an order topology on $\strop$ and a product topology on $\strop^{n}$. We also note that $\strop$ is homeomorphic to $\R$, with the homeomorphism given by the map 
\[
\strop \ni a \mapsto \begin{cases}
\sign(a)\exp(\abs{a}) &\text{if $a \neq \zero$} \\
0 &\text{otherwise.}
\end{cases}
\]
This also implies that $\strop^{n}$ is homeomorphic to $\R^{n}$. In particular, if $\bsalg \subset \puiseux^{n}$ is a semialgebraic set, then we can study the topological properties of its image $\sval(\bsalg) \subset \strop^{n}$. (We also note that the authors of \cite{jell_scheiderer_yu} do not use $\sval$, but its homeomorphic version given above. Since these maps differ by a homeomorphism, this choice does not impact the topological results.) Similarly to $\R^{n}$, the space $\strop^{n}$ has $2^{n}$ (closed) orthants, defined in the following way. If $\mu \in \{-1, +1\}^{n}$, then the \emph{$\mu$-orthant} of $\strop^{n}$ is defined as
\[
\torth^{n} \coloneqq \{x \in \strop^{n} \colon \forall k, \, \sign(x_{k}) \in \{0, \mu_{k}\} \} \, .
\]
We note that the orthants are closed in the topology of $\strop^{n}$. Moreover, they are homeomorphic to $\trop^{n}$, with the homeomorphism given by the absolute value map
\begin{equation}\label{eq:abs_value_map}
\torth \ni (x_{1} ,\dots, x_{n}) \mapsto (\abs{x_{1}}, \dots, \abs{x_{n}}) \in \trop^{n} \, .
\end{equation}
We can also divide $\strop^{n}$ into $3^{n}$ strata, in the following way. If $\sigma \in \{-1,0,+1\}^{n}$ is a vector of signs, then we define the \emph{$\sigma$-stratum} of $\strop^{n}$ as
\begin{equation}\label{eq:sigma_stratum}
\strop^{n, (\sigma)} \coloneqq \{x \in \strop^{n} \colon \forall k \in [n], \, \sign(x_{k}) = \sigma_{k} \} \, .
\end{equation}
We point out that in \cref{sec:model} we defined ``stratum'' as a set of points with only finite coordinates. However, in the present context it is more convenient to consider $\sigma$-strata as sets of points with some coordinates equal to $-\infty$, as in~\cref{eq:sigma_stratum}.
We note that the strata of $\strop^{n}$ are generally \emph{not} closed in the topology of $\strop^{n}$. More precisely, if $\sigma \neq 0$ and $\sum_{k} \abs{\sigma_{k}} = p$, then $\strop^{n, (\sigma)}$ is homeomorphic to $\postrop^{p}$ and the homeomorphism is given by the \emph{sign-forgetting projection} $\pi^{(\sigma)} \colon \strop^{n, (\sigma)} \to \postrop^{p}$ defined as $\pi^{(\sigma)}(x_{1}, \dots, x_{n}) \coloneqq (\abs{x_{k_{1}}}, \dots, \abs{x_{k_{p}}})$, where $\{k_{1}, \dots, k_{p}\}$ are the indices such that $\sigma_{k_{i}} \neq 0$. 

The following result about the topology of $\sval(\bsalg)$
follows from \cref{theorem:images_are_closed}, compare with \cite[Theorem~6.5]{jell_scheiderer_yu}.

\begin{corollary}\label[corollary]{corollary-compare}
If $\bsalg \subset \puiseux^{n}$ is a semialgebraic set, then $\sval(\bsalg) \cap \strop^{n,(\sigma)}$ is closed in the induced topology of $\strop^{n,(\sigma)}$ for every $\sigma \in \{-1,0,+1\}^{n}$. Moreover, if $\bsalg$ is closed, then $\sval(\bsalg)$ is closed in the topology of $\strop^{n}$.
\end{corollary}
\begin{proof}
To prove the first part, fix $\bsalg \subset \puiseux^{n}$ and $\sigma \in \{-1,0,+1\}^{n}$. If $\sigma = 0$, then the claim is trivial. Otherwise, consider the semialgebraic set $\bsalg_{\sigma} \coloneqq \{\bo x \in \bsalg \colon \forall k, \, \bo \sign(\bo x_{k}) = \sigma_{k} \}$ and note that $\sval(\bsalg) \cap \strop^{n,(\sigma)} = \sval(\bsalg_{\sigma})$. Denote $p \coloneqq \sum_{k} \abs{\sigma_{k}}$ and $K \coloneqq \{k \in [n] \colon \sigma_{k} \neq 0\}$. Let $\salg \coloneqq \val(\bsalg_{\sigma}) \subset \trop^{n}$ and let $\salg_{K} \subset \postrop^{p}$ denote the stratum of $\salg$ associated with $K$. By \cref{theorem:images_are_closed}, the set $\salg_{K}$ is closed in the topology of $\postrop^{p}$. Therefore, the claim follows from the fact that $\sval(\bsalg_{\sigma}) = (\pi^{(\sigma)})^{-1}(\salg_{K})$, where $\pi^{(\sigma)} \colon \strop^{n,(\sigma)} \to \postrop^{p}$ is the sign-forgetting projection. To prove the second part, fix a vector $\mu \in \{-1,+1\}^{n}$, let
\[
\puiseux^{n}_{\mu} \coloneqq \{\bo x \in \puiseux^{n} \colon \forall k, \, \sign(\bo x_{k}) \in \{0, \mu_{k}\} \}
\]
denote the corresponding orthant of the Puiseux series, and let $\bsalg_{\mu} \coloneqq \bsalg \cap \puiseux_{\mu}$. By \cref{theorem:images_are_closed}, the set $\val(\bsalg_{\mu})$ is closed in the topology of $\trop^{n}$. Moreover, $\sval(\bsalg_{\mu}) \subset \torth^{n}$ is the preimage of $\val(\bsalg_{\mu})$ under the absolute value map given in \cref{eq:abs_value_map}. Hence, $\sval(\bsalg_{\mu})$ is a closed subset of $\torth^{n}$ and therefore it is also a closed subset of $\strop^{n}$. The claim follows from the fact that $\sval(\bsalg) = \bigcup_{\mu \in \{-1,+1\}^{n}} \sval(\bsalg_{\mu})$.
\end{proof}

\section{Tropical spectrahedra}\label[section]{section:spectrahedra}

\subsection{Tropicalization of nonarchimedean spectrahedra} We now introduce the notion of tropical spectrahedra.
\begin{definition}
A set $\spectra \subset \trop^n$ is said to be a \emph{tropical spectrahedron} if there exists a spectrahedron $\bspectra \subset \nnpuiseux^{n}$ such that $\spectra = \val(\bspectra)$.
\end{definition}
If $\spectra = \val(\bspectra)$, then we refer to $\spectra$ as the \emph{tropicalization} of the spectrahedron $\bspectra$, and $\bspectra$ is said to be a \emph{lift (over the field $\puiseux$)} of~$\spectra$. 

Recall that we have the following characterization of positive semidefinite matrices:
\begin{proposition}\label[proposition]{prop:semidefinite_by_minors}
A symmetric matrix $\bo A \in \puiseux^{m \times m}$ is positive semidefinite if and only if every principal minor of $\bo A$ is nonnegative.
\end{proposition}
Given symmetric matrices $\bo Q^{(0)}, \dots, \bo Q^{(n)} \in \puiseux^{m \times m}$ and $\bo x \in \puiseux^n$, we denote by $\bo Q(\bo x)$ the matrix pencil $\bo Q^{(0)} + \bo x_1 \bo Q^{(1)}+ \dots + \bo x_n \bo Q^{(n)}$. \Cref{prop:semidefinite_by_minors} provides a description of the spectrahedron $\bspectra = \{ \bo x \in \nnpuiseux^n \colon \bo Q(\bo x) \loew 0 \}$ by a system of polynomial inequalities of the form $\det \bo Q_{I \times I}(\bo x) \geq 0$, where $I$ is a nonempty subset of $[m]$, and $\det \bo Q_{I \times I}(\bo x)$ corresponds to the $(I \times I)$-minor of the matrix $\bo Q(\bo x)$. Following this, we obtain that the tropical spectrahedron~$\spectra$ is included in the intersection of the sets $\{ x \in \trop^n \colon \tropPol(\bo P)^+(x) \geq \tropPol(\bo P)^-(x) \}$ where $\bo P$ is a polynomial of the form $\det \bo Q_{I \times I}(\bo x)$. In general, this inclusion may be strict. We refer to~\cite[Example~15]{tropical_simplex} for an example in which $\bspectra$ is a polyhedron. Nevertheless, under the regularity assumption stated in \cref{corollary:purity_commutation}, both sets coincide. In fact, we prove that, under similar assumptions, tropical spectrahedra have a description that is much simpler than the one provided by \cref{corollary:purity_commutation}. This description only involves principal tropical minors of order $2$. 

Our results are divided into three parts. In \cref{subsubsec:tropical_metzler_spectrahedra} we deal with spectrahedra defined by Metzler matrices $\bo Q^{(0)}, \dots, \bo Q^{(n)}$ (i.e., matrices in which the off-diagonal entries are nonpositive). 
This enables us to use a lemma that is similar to \cref{corollary:purity_commutation} in order to give a description of tropical spectrahedra under a regularity assumption.

In \cref{subsection:nonmetzler} we switch to non-Metzler matrices. In this case, tropical spectrahedra 
may not 
be regular, even under strong genericity assumptions. Nevertheless, we are able to extend our previous analysis to this case and give a description, involving only principal minors of size at most~$2$, of non-Metzler spectrahedra, under a regularity assumption over some associated sets.

Finally, the purpose of \cref{section:generic} is to show that the regularity assumptions used in \cref{subsubsec:tropical_metzler_spectrahedra,subsection:nonmetzler} hold generically.

Let us start with some introductory remarks. First, observe that in order to characterize the class of tropical spectrahedra, it is enough to restrict ourselves to tropical spectrahedral cones, as the image of a spectrahedron can be deduced from the image of its homogenized version. This is formally stated in the next lemma.
\begin{lemma}
Let $\bo Q^{(0)}, \dots, \bo Q^{(n)} \in \puiseux^{m \times m}$ be a sequence of symmetric matrices. Define
\[
\bspectra \coloneqq \{ \bo x \in \nnpuiseux^{n} \colon \bo Q^{(0)} + \bo x_{1} \bo Q^{(1)}+ \dots + \bo x_{n} \bo Q^{(n)} \loew 0 \} \,
\]
and
\[
\hombspectra \coloneqq \{ (\bo x_0, \bo x) \in \nnpuiseux^{n+1} \colon \bo x_0 \bo Q^{(0)} + \bo x_{1} \bo Q^{(1)}  + \dots + \bo x_{n} \bo Q^{(n)} \loew 0 \} \, .
\]
Then
\[
\val(\bspectra) = \pi(\{ x \in \val(\hombspectra) \colon x_0 = 0 \}) \, ,
\]
where $\pi \colon \trop^{n+1} \to \trop^{n}$ denotes the projection that forgets the first coordinate.
\end{lemma}
\begin{proof}
We start by proving the inclusion $\subset$. Take any $x \in \val(\bspectra)$ and its lift $\bo x \in \bspectra \cap \val^{-1}(x)$. Observe that the point $(1, \bo x)$ belongs to $\hombspectra$. Therefore, the point $(0, x)$ belongs to $\val(\hombspectra)$ and $x$ belongs to $\pi(\{ x \in \val(\hombspectra) \colon x_0 = 0 \})$. Conversely, let $x$ belong to $\pi(\{ x \in \val(\hombspectra) \colon x_0 = 0 \})$. Then $(0,x)$ belongs to $\val(\hombspectra)$. In other words, there exists a lift $(\bo z, \bo x) \in \hombspectra$ such that $\val(\bo z) = 0$ and $\val(\bo x) = x$. Take the point $(1, \bo x/\bo z)$. This point also belongs to $\hombspectra$. Moreover, $\bo x / \bo z$ belongs to $\bspectra$. Hence, the point $x = \val(\bo x / \bo z)$ belongs to $\val(\bspectra)$.
\end{proof}

Second, let us explain our approach to the tropicalization of spectrahedra. It relies on the next elementary lemma. The latter should
be compared with
a result of Yu~\cite{yu_semidefinite_cone} that shows
that the set of images by the valuation of the set of positive semidefinite
matrices $\bo A$ over the field of Puiseux series is determined by the inequalities $\val(\bo A_{ii}) +  \val (\bo A_{jj}) \geq 2 \val (\bo A_{ij})$ for $i\neq j$. 

\begin{lemma}\label[lemma]{lemma:dominant_minors_of_order_two}
Let $\bo A \in \puiseux^{m \times m}$ be a symmetric matrix.
Suppose that $\bo A$ has nonnegative entries on its diagonal and that the inequality $\bo A_{ii}\bo A_{jj} \ge (m-1)^{2}\bo A_{ij}^{2}$ holds for all pairs $(i,j)$ such that $i \neq j$. Then $\bo A$ is positive semidefinite.
\end{lemma}
\begin{proof}
If $\bo A$ is a zero matrix, then there is nothing to show. From now on we suppose that $\bo A$ has at least one nonzero entry. First, let us suppose that $\bo A$ has positive entries on its diagonal. In this case, let $\bo B \in \puiseux^{m \times m}$ be the diagonal matrix defined by $\bo B_{ii} \coloneqq \bo A_{ii}^{-1/2}$ for all $i$. Observe that $\bo A$ is positive semidefinite if and only if the matrix $\bo D \coloneqq \bo B \bo A \bo B$ is positive semidefinite. Moreover, $\bo D$ has ones on its diagonal and $\bo D_{ij} = \bo A_{ii}^{-1/2} \bo A_{ij} \bo A_{jj}^{-1/2}$ for all $i \neq j$. Hence $\abs{\bo D_{ij}} \le 1/(m-1)$ for all $i \neq j$. Therefore $\bo D$ is diagonally dominant and hence positive semidefinite.

Second, if $\bo A$ has some zeros on its diagonal, let $I = \{i \in [m] \colon \bo A_{ii} \neq 0\}$. Since the inequality $\bo A_{ii}\bo A_{jj} \ge (m-1)^{2}\bo A_{ij}^{2}$ holds, we have $\bo A_{ij} = 0$ if either $i \notin I$ or $j \notin I$. Let $\bo A_{I}$ denote the submatrix formed by the rows and columns with indices from $I$. Then $\bo A$ is positive semidefinite if and only if $\bo A_{I}$ is positive semidefinite. Finally, $\bo A_{I}$ is positive semidefinite by the considerations from the previous paragraph.
\end{proof}
Given a spectrahedron $\bspectra= \{\bo x \in \nnpuiseux^{n} \colon \bo Q(\bo x) \loew 0\}$ we define two sets $\outbspectra$, $\inbspectra \subset \nnpuiseux^n$ as
\begin{align*}
\outbspectra &\coloneqq
\Bigl\{ \bo x \in \nnpuiseux^{n} \colon  \forall i, \bo Q_{ii}(\bo x) \ge 0 \, ,
 \forall i \neq j, \bo Q_{ii}(\bo x)\bo Q_{jj}(\bo x) \ge (\bo Q_{ij}(\bo x))^{2} \Bigr\} \, , 
 \\
\inbspectra &\coloneqq \Bigl\{ \bo x \in \nnpuiseux^{n} \colon \forall i, \bo Q_{ii}(\bo x) \ge 0 \, ,
\forall i \neq j, \bo Q_{ii}(\bo x)\bo Q_{jj}(\bo x) \ge (m-1)^2 (\bo Q_{ij}(\bo x))^{2} \Bigr\} \, .
\end{align*}

\Cref{prop:semidefinite_by_minors} shows that $\bspectra \subset \outbspectra$, while \cref{lemma:dominant_minors_of_order_two} shows that $\inbspectra \subset \bspectra$. In order to describe the set $\val(\bspectra)$, we will exhibit conditions that ensure that the tropicalizations of $\inbspectra$ and $\outbspectra$ coincide, i.e., $\val(\outbspectra) = \val(\bspectra) = \val(\inbspectra)$. 

\subsection{Tropical Metzler spectrahedra}\label[section]{subsubsec:tropical_metzler_spectrahedra}

In this section, we study the spectrahedra that are defined by Metzler matrices. Recall that a square matrix $\bo A \in \puiseux^{m \times m}$ is a
\emph{(negated) Metzler matrix} if its off-diagonal coefficients are nonpositive. Similarly, we say that a matrix $M \in \strop^{m \times m}$ is a \emph{tropical Metzler matrix} if $M_{ij} \in \negtrop \cup \{\zero\}$ for all $i \neq j$. Let $Q^{(1)}, \dots, Q^{(n)} \in \strop^{m \times m}$ be symmetric tropical Metzler matrices. Given $i, j \in [m]$, we refer to $Q_{ij}(X)$ as the tropical polynomial:
\[
Q_{ij}(X) \coloneqq Q^{(1)}_{ij} \tdot X_{1} \tplus \cdots \tplus Q^{(n)}_{ij} \tdot X_{n} \, .
\]
 
\begin{definition}\label[definition]{def:metzler_spectrahedron}
If $Q^{(1)}, \dots, Q^{(n)} \in \strop^{m \times m}$ are symmetric tropical Metzler matrices, we define the \emph{tropical Metzler spectrahedron} $\spectra(Q^{(1)}, \dots, Q^{(n)})$ described by $Q^{(1)}, \dots, Q^{(n)}$ as the set of points $x \in \trop^{n}$ that fulfill the following two conditions:
\begin{compactitem}
\item for all $i \in [m]$, $Q_{ii}^{+}(x) \ge Q_{ii}^{-}(x)$;
\item for all $i,j \in [m]$, $i <j$, $Q_{ii}^{+}(x) \tdot Q_{jj}^{+}(x) \ge (Q_{ij}(x))^{\tdot 2}$.
\end{compactitem}
\end{definition}
Observe that the term $Q_{ij}(x)$ ($i \neq j$) is well defined for any $x \in \trop^n$ thanks to the Metzler property of the matrices $Q^{(k)}$. 

Where there is no ambiguity, we denote $\spectra(Q^{(1)}, \dots, Q^{(n)})$ by~$\spectra$. With standard notation, the constraints defining this set respectively read: for all $i \in [m]$, 
\begin{equation}
\max_{Q^{(k)}_{i i} \in \postrop} \bigl(Q^{(k)}_{i i} + x_k \bigr)
\geq 
\max_{Q^{(l)}_{i i} \in \negtrop} \bigl(\abs{Q^{(l)}_{i i}} + x_l \bigr)\, ,\label[equation]{eq:first_kind}
\end{equation}
and for all $i, j \in [m]$ such that $i < j$,
\begin{equation}
\begin{multlined}
\max_{Q^{(k)}_{i i} \in \postrop} \bigl(Q^{(k)}_{i i} + x_k\bigr) + \max_{Q^{(k')}_{j j} \in \postrop} \bigl(Q^{(k')}_{j j} + x_{k'} \bigr)
\geq 2 \max_{l \in [n]} \bigl(\abs{Q^{(l)}_{i j}} + x_l \bigr)\, .
\end{multlined} \label[equation]{eq:second_kind}
\end{equation}

\begin{remark}\label[remark]{remark-game}
  The system of constraints~\cref{eq:first_kind} and~\cref{eq:second_kind} has an interpretation in terms of games.
  Indeed, it is shown in~\cite{issac2016jsc,mega2017} to be equivalent
  to an equation of the form $x\leq T(x)$, where $T$ is the dynamic
  programming operator of a turn-based stochastic mean payoff game.
  The existence
  of a solution $x\in \trop^n$, not identically $-\infty$, is
  equivalent to the existence of an initial state that is winning.
  See also~\cite{1802.07712,skomra:tel-01958741} for
  more information on the relation between tropical spectrahedra
  and turn-based games.
  \end{remark}
The next proposition justifies the terminology introduced in \cref{def:metzler_spectrahedron}, and ensures that the set $\spectra$ is indeed a tropical spectrahedron. To this end, we explicitly construct a spectrahedron $\bspectra \subset \nnpuiseux^{n}$ verifying $\val(\bspectra) = \spectra$. (This result was already announced in the conference version of the paper~\cite{issac2016jsc}. We provide here the full proof for the sake of completeness.)
\begin{proposition}\label[proposition]{proposition:tropical_metzler_spectrahedron}
The set $\spectra(Q^{(1)}, \dots, Q^{(n)})$ is a tropical spectrahedron.
\end{proposition}
\begin{proof}
Let us define the matrices $\bo Q^{(1)}, \dots, \bo Q^{(n)} \in \puiseux^{m \times m}$ as follows: 
\begin{itemize}
\item if $Q^{(k)}_{ij} \in \negtrop$, then we set $\bo Q^{(k)}_{ij} \coloneqq -t^{|Q^{(k)}_{ij}|}$;
\item if $Q^{(k)}_{ij} \in \postrop$ (which, under our assumptions, can happen only if $i = j$), then $\bo Q^{(k)}_{ij} \coloneqq m n t^{Q^{(k)}_{ij}}$;
\item if $Q^{(k)}_{ij} = \zero$, then $\bo Q^{(k)}_{ij} \coloneqq 0$.
\end{itemize}
Consider the spectrahedron $\bspectra \coloneqq \{\bo x \in \nnpuiseux^n \colon \bo Q(\bo x) \loew 0\}$. We claim that $\val(\bspectra) = \spectra$. 

We start with the inclusion $\val(\outbspectra) \subset \spectra$. Let $\bo x \in \outbspectra$. Observe that for all $i \neq j$, the inequality $\bo Q^+_{ii}(\bo x)\bo Q^+_{jj}(\bo x) \ge (\bo Q_{ij}(\bo x))^{2}$ holds thanks to the fact that $\bo Q_{ii} (\bo x)\geq 0$. Moreover, we have $\val(\bo Q_{ii}^{+}(\bo x)) = Q_{ii}^{+}(x)$, where $x = \val(\bo x)$. Similarly, $\val(\bo Q_{ii}^{-}(\bo x)) = Q_{ii}^{-}(x)$. As the $Q^{(k)}$ are tropical Metzler matrices, we have $\val(\bo Q_{ij}(\bo x)) = \abs{Q_{ij}(x)}$ for $i \neq j$. Since the map $\val$ is order preserving over $\nnpuiseux$, we deduce that $x \in \spectra$.

Now, let us prove the inclusion $\spectra\subset \val(\inbspectra)$. Take any $x \in \spectra$ and its lift $\bo x_k = t^{x_k}$, with the convention that $t^{\zero} = 0$. First, as noted in the previous paragraph, we have $\val(\bo Q_{ii}^{+}(\bo x)) = Q_{ii}^{+}(x)$ and $\val(\bo Q_{ii}^{-}(\bo x)) = Q_{ii}^{-}(x)$. We have chosen the matrices $\bo Q^{(k)}$ and the point $\bm x$ in such a way that 
\begin{equation}
\bo Q_{ii}^+(\bo x) = \sum_{Q_{ii}^{(k)} \in \trop^+} m n t^{Q^{(k)}_{i i} + x_k} \geq m n t^{Q_{ii}^{+}(x)} \, . \label[equation]{eq:ineq_Qii}
\end{equation}
Similarly, we have $
\bo Q_{ii}^-(\bo x) = \sum_{Q_{ii}^{(k)} \in \trop^-} t^{\abs{Q^{(k)}_{i i}} + x_k} \leq n t^{Q_{ii}^-(x)}$. 
Since $Q_{ii}^{+}(x) \geq Q_{ii}^-(x)$, we deduce that $\bo Q^{-}_{ii}(\bo x) \le \frac{1}{m}\bo Q^{+}_{ii}(\bo x)$, and so $\bo Q_{ii}(\bo x)  \ge (1 - \frac{1}{m}) \bo Q^{+}_{ii}(\bo x) \ge 0$. Second, for all $i \neq j$ we have $0 \geq \bo Q_{ij}(\bo x) \geq -n t^{\abs{Q_{ij}(x)}}$. Using~\cref{eq:ineq_Qii} and the fact that $Q_{ii}^{+}(x) \tdot Q_{jj}^{+}(x) \ge (Q_{ij}(x))^{\tdot 2}$, we obtain 
\[ \bo Q_{ij}(\bo x)^{2}  \le \frac{1}{m^2}\bo Q^{+}_{ii}(\bo x) \bo Q^{+}_{jj}(\bo x)\,.
\]
Therefore, by the previous inequalities, 
\begin{equation*}
\begin{aligned}
\bo Q_{ii}(\bo x) \bo Q_{jj}(\bo x) &- (m-1)^2\bo Q_{ij}(\bo x)^{2} \\ &\ge \Bigl(1 - \frac{1}{m}\Bigr)^{2}\bo Q^{+}_{ii}(\bo x) \bo Q^{+}_{jj}(\bo x) - (m-1)^2\bo Q_{ij}(\bo x)^{2} \ge 0 \, .
\end{aligned}
\end{equation*}
Hence $\bo x \in \inbspectra$. Therefore, by \cref{prop:semidefinite_by_minors,lemma:dominant_minors_of_order_two} we have $\val(\bspectra) \subset \val(\outbspectra) \subset \spectra \subset \val(\inbspectra) \subset \val(\bspectra)$, which implies that $\val(\bspectra) = \spectra$.
\end{proof}

\begin{figure}[t]
\centering
\begin{tikzpicture}
\begin{scope}[scale = 0.6]
      \draw[gray!60, ultra thin] (0,0) grid (9,9);
      \draw[gray!80!black, thin, ->] (0,0) -- (9,0) node[below right] {$x_1$};
      \draw[gray!90!black, thin, ->] (0,0) -- (0,9) node[below left] {$x_2$};
       \fill[fill=lightgray, fill opacity = 0.7]
        (1,4) -- (3, 4) -- (4,3) -- (4,1) -- (5,1) -- (5,2) -- (6,4) -- (8,6) -- (8, 8) -- (6,8) -- (5,7) -- (3,6) -- (1,6) -- cycle;
        \draw[very thick] (1,4) -- (3, 4) -- (4,3) -- (4,1) -- (5,1) -- (5,2) -- (6,4) -- (8,6) -- (8, 8) -- (6,8) -- (5,7) -- (3,6) -- (1,6) -- cycle;
\end{scope}
\end{tikzpicture}
\caption{A tropical Metzler spectrahedron.}\label[figure]{fig:tropical_metzler}
\end{figure}
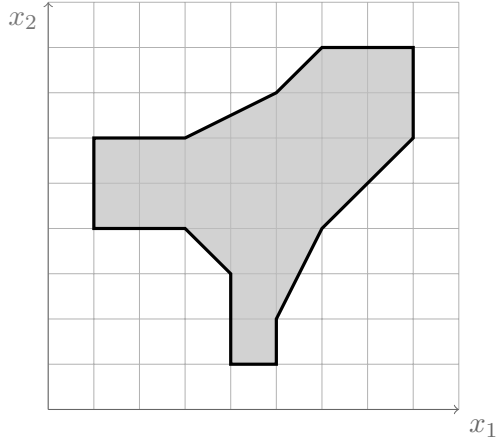

\begin{example}
If $A^{(1)}, \dots, A^{(p)}$ are matrices, then $\diag(A^{(1)}, \dots, A^{(p)})$ refers to the block diagonal matrix with blocks $A^{(s)}$ on the diagonal and all other entries equal to $-\infty$. Let $Q^{(0)}, Q^{(1)}, Q^{(2)} \in \strop^{9 \times 9}$ be symmetric tropical Metzler matrices defined as follows:
\begin{equation*}
\begin{aligned}
Q^{(0)} &\coloneqq \diag \Bigl( 8, \tminus 1, \tminus 1,
\begin{bmatrix}
\zero & \tminus 3 \\
\tminus 3 & \zero
\end{bmatrix},
\begin{bmatrix}
2 & \zero \\
\zero & 8
\end{bmatrix},
\begin{bmatrix}
3 & \zero \\
\zero & 9
\end{bmatrix} 
\Bigr) \, , \\
Q^{(1)} &\coloneqq \diag \Bigl(\tminus 0, 0, \zero,
\begin{bmatrix}
0 & \zero \\
\zero & -2
\end{bmatrix},
\begin{bmatrix}
\zero & \tminus 0 \\
\tminus 0 & \zero
\end{bmatrix},
\begin{bmatrix}
0 & \zero \\
\zero & 4
\end{bmatrix}
\Bigr) \, , \\
Q^{(2)} &\coloneqq \diag \Bigl( \tminus 0, \zero, 0,
\begin{bmatrix}
-1 & \zero \\
\zero & -1
\end{bmatrix},
\begin{bmatrix}
0 & \zero \\
\zero & 4
\end{bmatrix},
\begin{bmatrix}
\zero & \tminus 0 \\
\tminus 0 & \zero
\end{bmatrix}
\Bigr) \, .
\end{aligned}
\end{equation*}
The intersection of the tropical Metzler spectrahedron $\spectra(Q^{(0)}, Q^{(1)}, Q^{(2)})$ with the hyperplane $\{x_{0} = 0 \}$ is depicted in \cref{fig:tropical_metzler}.
\end{example}

We now focus on the main problem of characterizing the image by the valuation of a spectrahedron defined by Metzler matrices. Our goal is to show that any spectrahedron $\bspectra = \{ \bo x \in \nnpuiseux^n \colon \bo Q(\bo x) \loew 0 \}$ verifying 
$\sval(\bo Q^{(k)}) = Q^{(k)}$ is mapped to the tropical Metzler spectrahedron $\spectra$, provided that some assumptions related to the genericity of the matrices $Q^{(k)}$ and the regularity of the set $\spectra$ hold. To do so, we prove a weaker result, \cref{th:metzler_positive_orthant}, on the tropicalization of the spectrahedron restricted to the open positive orthant $\pospuiseux^{n}$.
\begin{lemma}\label[lemma]{lemma:positive_definite_lift}
Let $A \in \trop^{m \times m}$ be a symmetric matrix such that $A_{ii} \in \postrop \cup \{\zero \}$ for all $i$ and $A_{ii} \tdot A_{jj} > A_{ij}^{\tdot 2}$ for all $i < j$ such that $A_{ij} \neq \zero$. Let $\bo A \in \puiseux^{m \times m}$ be any symmetric matrix such that $\sval(\bo A) = A$. Then $\bo A$ fulfills the conditions of \cref{lemma:dominant_minors_of_order_two}. (In particular, it is positive semidefinite.)
\end{lemma}
\begin{proof}
Since $A_{ii} \in \postrop \cup \{\zero \}$ for all $i$, we have $\bo A_{ii} \ge 0$ for all $i$. Moreover, if $A_{ij} = \zero$, then $\bo A_{ii} \bo A_{jj} \ge 0$ and if $A_{ij} \neq \zero$, then $\val(\bo A_{ii} \bo A_{jj})  = A_{ii} \tdot A_{jj} > A_{ij}^{\tdot 2} = \val((m - 1)^{2} \bo A^{2}_{ij})$. Therefore $\bo A$ fulfills the conditions of \cref{lemma:dominant_minors_of_order_two}.
\end{proof}
\begin{lemma}\label[lemma]{lemma:inclusion_of_interior}
Let $\interspectra$ be the set of points $x \in \R^{n}$ that fulfill the following two conditions:
\begin{compactitem}
\item for all $i \in [m]$ such that $Q_{ii}^{-}$ is nonzero we have $Q_{ii}^{+}(x) > Q_{ii}^{-}(x)$;
\item for all $i,j \in [m]$, $i <j$, such that $Q_{ij}$ is nonzero we have $Q_{ii}^{+}(x) \tdot Q_{jj}^{+}(x) > (Q_{ij}(x))^{\tdot 2}$.
\end{compactitem}
Then $\cl(\interspectra) \subset \val(\inbspectra \cap \pospuiseux^{n})$ for every spectrahedron $\bspectra = \{ \bo x \in \nnpuiseux^n \colon \bo Q(\bo x) \loew 0 \}$ such that $\sval(\bo Q^{(k)}) = Q^{(k)}$.
\end{lemma}
\begin{proof}
By \cref{theorem:images_are_closed}, the set $\val(\inbspectra \cap \pospuiseux^{n})$ is closed. Therefore, it is enough to prove that $\interspectra \subset \val(\inbspectra)$. Fix any $x \in \interspectra$ and take any lift $\bo x \in \val^{-1}(x) \cap \nnpuiseux^{n}$. Let $A \coloneqq \sval(\bo Q(\bo x))$. For any $i$ such that $Q_{ii}^{-}$ is nonzero we have $\val(\bo Q^{+}_{ii}(\bo x)) = Q_{ii}^{+}(x) > Q_{ii}^{-}(x) = \val(\bo Q^{-}_{ii}(\bo x))$. Therefore $A_{ii} = \sval(\bo Q_{ii}(\bo x)) = Q_{ii}^{+}(x) \in \postrop \cup \{ \zero\}$ for all $i$ (even if $Q_{ii}^{-}$ is zero). Furthermore, we have $A_{ij} = Q_{ij}(x)$ for any $i < j$. Therefore, for any $i < j$ such that $A_{ij} \neq -\infty$, we have $A_{ii} \tdot A_{jj} > A_{ij}^{\tdot 2}$. Hence, by \cref{lemma:positive_definite_lift}, $\bo Q(\bo x)$ fulfills the conditions of \cref{lemma:dominant_minors_of_order_two}. In other words, $\bo x \in \inbspectra$ and $x \in \val(\inbspectra \cap \pospuiseux^{n})$.
\end{proof}

\begin{assumption}\label[assumption]{assumption:nondeg_minors}
We suppose that for every matrix $Q^{(k)}$ and every pair $i \neq j$  such that $Q^{(k)}_{ii}$ and $Q^{(k)}_{jj}$ belong to $\postrop$ the inequality $Q^{(k)}_{ii} + Q^{(k)}_{jj} \neq 2\abs{Q^{(k)}_{ij}}$ holds.
\end{assumption}
We point out that \cref{assumption:nondeg_minors} can be interpreted in terms of the nonsingularity of some (tropical) minors of order $2$ of the matrices $Q^{(k)}$. 

\begin{theorem}\label[theorem]{th:metzler_positive_orthant}
Let $\bspectra = \{ \bo x \in \nnpuiseux^n \colon \bo Q(\bo x) \loew 0 \}$ be a spectrahedron described by Metzler matrices $\bo Q^{(1)}$, \dots, $\bo Q^{(n)}$ such that $\sval(\bo Q^{(k)}) = Q^{(k)}$. Suppose that \cref{assumption:nondeg_minors} holds and that the set $\spectra(Q^{(1)}, \dots, Q^{(n)}) \cap \R^{n}$ is regular. Then 
\[
\val(\bspectra \cap \pospuiseux^{n}) = \spectra(Q^{(1)}, \dots, Q^{(n)}) \cap \R^{n} \, .
\]
\end{theorem}
\begin{proof}
Let $\interspectra$ be defined as in \cref{lemma:inclusion_of_interior}. Using the same arguments as in the proof of \cref{proposition:tropical_metzler_spectrahedron}, we can show that $\val(\outbspectra) \subset \spectra$, and subsequently, $\val(\outbspectra \cap \pospuiseux^{n}) \subset \spectra \cap \R^{n}$. Then, by \cref{lemma:inclusion_of_interior}, it is enough to show that $\cl(\interspectra) = \spectra \cap \R^{n}$. Observe that the inequalities defining $\spectra$ such that $Q^-_{ij}$ ($i \leq j$) is the zero tropical polynomial are trivially satisfied. Therefore, $\spectra$ can be expressed as the set of points $x \in \trop^n$ verifying:
\begin{itemize}
\item for all $i \in [m]$ such that $Q^-_{ii}$ is nonzero, $Q^+_{ii}(x) \geq Q^-_{ii}(x)$;
\item for all $i, j \in [m]$, $i < j$, such that $Q_{ij}$ (or, equivalently, $Q^-_{ij}$) is nonzero,  $Q_{ii}^{+}(x) \tdot Q_{jj}^{+}(x) \ge (Q_{ij}(x))^{\tdot 2}$.
\end{itemize}
We denote by $\Xi$ the set of $(i,j) \in [m] \times [m]$ such that $i \leq j$ and $Q^-_{ij}$ is nonzero. Since $\spectra \cap \R^{n}$ is supposed to be regular, we propose to use \cref{lemma:purity_closes_strict}, and thus, to exhibit nonzero tropical polynomials $P_{ij}$ such that
\begin{align}
\spectra \cap \R^{n} = \{x \in \R^{n} \colon \forall (i,j) \in \Xi, P_{ij}^{+}(x) \ge P_{ij}^{-}(x) \} \, , \label[equation]{eq:closure} \\
\shortintertext{and}
\interspectra = \{x \in \R^{n} \colon \forall (i,j) \in \Xi, P_{ij}^{+}(x) > P_{ij}^{-}(x) \} \, . \label[equation]{eq:interior}
\end{align}
In other words, we want to express the inequalities of the form $Q_{ii}^{+}(x) > Q_{ii}^{-}(x)$ and $Q_{ii}^{+}(x) \tdot Q_{jj}^{+}(x) > (Q_{ij}(x))^{\tdot 2}$ as tropical polynomial inequalities in which no term appears both on the left- and on the right-hand side. The inequalities of the first kind already satisfy this condition, and it suffices to set $P_{ii} \coloneqq Q_{ii}$ for all $i \in [m]$ such that $(i,i) \in \Xi$. In contrast, we have to transform the inequalities of the second kind into equivalent contraints of the form $P_{ij}^{+}(x) > P_{ij}^{-}(x)$, where $(i,j) \in \Xi$ and $i < j$. To this end, we use \cref{assumption:nondeg_minors}. First, observe that the functions $(Q_{ij}(x))^{\tdot 2}$ and $\tsum_{k}(Q^{(k)}_{ij} \tdot x_{k})^{\tdot 2}$ are equal. Therefore, we can replace the inequalities $Q_{ii}^{+}(x) \tdot Q_{jj}^{+}(x) > (Q_{ij}(x))^{\tdot 2}$ by $Q_{ii}^{+}(x) \tdot Q_{jj}^{+}(x) > \tsum_{k}(Q^{(k)}_{ij} \tdot x_{k})^{\tdot 2}$. Now, we can define a formal subtraction of these tropical expressions. More precisely, for every $(i,j) \in \Xi$ such that $i < j$, we define 
\[
P_{ij} \coloneqq \biggl(\tsum_{\substack{k \neq l\\ Q_{ii}^{(k)}, \,  Q_{jj}^{(l)} \in \postrop}} (Q_{ii}^{(k)} \tdot Q_{jj}^{(l)}) \tdot (X_k \tdot X_l) \biggr) \tplus \biggl(\tsum_{\substack{Q_{ii}^{(k)}, \, Q_{jj}^{(k)} \in \postrop\\\text{or}\; \abs{Q^{(k)}_{ij}} \neq -\infty}} \alpha_k \tdot X_k^{\tdot 2} \biggr) \, ,
\]
where $\alpha_k$ is given by:
\[
\alpha_k \coloneqq
\begin{cases}
Q_{ii}^{(k)} \tdot Q_{jj}^{(k)} & \text{if} \; Q_{ii}^{(k)}, Q_{jj}^{(k)} \in \postrop \; \text{and} \; Q^{(k)}_{ii} + Q^{(k)}_{jj} > 2\abs{Q^{(k)}_{ij}}\, , \\
\tminus (Q^{(k)}_{ij})^{\tdot 2} & \text{otherwise}. 
\end{cases}
\]
Recall that any inequality of the form $\max(x,\alpha +y) > \max(x', \beta + y)$ is equivalent to $\max(x, \alpha + y) > x'$ if $\alpha > \beta$, and to $x > \max(x', \beta + y)$ if $\beta > \alpha$. Therefore, \cref{assumption:nondeg_minors} ensures that $P_{ij}^{+}(x) > P_{ij}^{-}(x)$ is equivalent to $Q_{ii}^{+}(x) \tdot Q_{jj}^{+}(x) > \tsum_{k}(Q^{(k)}_{ij} \tdot  x_{k})^{\tdot 2}$. The same applies to the nonstrict counterparts of these inequalities. We conclude that \cref{eq:closure,eq:interior} are satisfied.
\end{proof}

\begin{theorem}\label[theorem]{corollary:metzler_regularity}
Let $\bspectra = \{ \bo x \in \nnpuiseux^n \colon \bo Q(\bo x) \loew 0 \}$ be a spectrahedron described by Metzler matrices $\bo Q^{(1)}$, \dots, $\bo Q^{(n)}$ such that $\sval(\bo Q^{(k)}) = Q^{(k)}$. Suppose that \cref{assumption:nondeg_minors} is fulfilled and that every stratum of the set 
$\spectra(Q^{(1)}, \dots, Q^{(n)})$ is regular. Then
\[
\val(\bspectra) = \spectra(Q^{(1)}, \dots, Q^{(n)}) \, .
\]
\end{theorem}
\begin{proof}
Fix a nonempty subset $K \subset [n]$. Observe that the stratum of $\val(\bspectra)$ associated with $K$ is equal to $\val(\bspectra^{(K)} \cap \pospuiseux^{K})$, where $\bspectra^{(K)}$ is the spectrahedron described by $(\bo Q^{(k)})_{k \in K}$. Similarly, the stratum of $\spectra$ associated with $K$ is equal to $\spectra^{(K)} \cap \R^{K}$, where $\spectra^{(K)}$ denotes the tropical Metzler spectrahedron described by $(Q^{(k)})_{k \in K}$. Therefore, we obtain the claim by applying \cref{th:metzler_positive_orthant} to every stratum.
\end{proof}

\subsection{Non-Metzler spectrahedra}\label[section]{subsection:nonmetzler}

In this section, we abandon the Metzler assumption that was imposed in the previous section. 
Let $Q^{(1)}$, \dots, $Q^{(n)} \in \strop^{m \times m}$ be symmetric tropical matrices. We introduce the set $\spectra(Q^{(1)}, \dots, Q^{(n)})$ (or simply $\spectra$) of points $x \in \trop^{n}$ that fulfill the following two conditions:
\begin{itemize}
\item for all $i \in [m]$, $Q_{ii}^{+}(x) \ge Q_{ii}^{-}(x)$;
\item for all $i,j \in [m]$, $i <j$, we have $Q_{ii}^{+}(x) \tdot Q_{jj}^{+}(x) \ge (Q^+_{ij}(x) \tplus Q^-_{ij}(x))^{\tdot 2}$ or $Q_{ij}^{+}(x) = Q_{ij}^{-}(x)$.
\end{itemize}
We point out that this generalizes \cref{def:metzler_spectrahedron} to the case of non-Metzler matrices. 
More precisely, if the matrices $Q^{(1)}$, \dots, $Q^{(n)} \in \strop^{m \times m}$ are Metzler, then $Q_{ij}^{+}(x) = -\infty$ and $(Q^+_{ij}(x) \tplus Q^-_{ij}(x))^{\tdot 2} = (Q_{ij}(x))^{\tdot 2}$ for all $i \neq j$ and $x \in \trop^{n}$. However, if the matrices $Q^{(k)}$ are not Metzler, then we have two possibilities. If $Q_{ij}$ does not vanish on $x$, then $(Q^+_{ij}(x) \tplus Q^-_{ij}(x))^{\tdot 2} = (Q_{ij}(x))^{\tdot 2}$ and we recover the same kind of inequality as in the Metzler case. On the other hand, if $Q_{ij}$ vanishes on $x$, then $Q_{ij}^{+}(x) = Q_{ij}^{-}(x)$ and we ignore the corresponding inequality.

We do \emph{not} claim that the set $\spectra$ defined above is a tropical spectrahedron. 
In this work we only show that this is true under some additional assumptions (which are generically fulfilled as shown in \cref{section:generic}). First, we need some notation. For every subset
\[
\Sigma \subset \{(i,j) \in [m]^{2} \colon i < j\}
\]
we denote
\[
\Sigma^{\complement} \coloneqq \{ (i,j) \in [m]^{2} \colon i < j, (i,j) \notin \Sigma \} \, .
\]
For every $\Sigma$ and every $\Diamond \in \{\le, \ge \}^{\Sigma^{\complement}}$ we define $\spectra_{\Sigma, \Diamond} (Q^{(1)}, \dots, Q^{(n)})$ (or $\spectra_{\Sigma, \Diamond}$ for short) as the set of all $x \in \trop^{n}$ such that
\begin{itemize}
\item for all $i \in [m]$, $Q_{ii}^{+}(x) \ge Q_{ii}^{-}(x)$;
\item for all $i,j \in [m]$, $i <j$, $(i,j) \in \Sigma$, $Q_{ii}^{+}(x) \tdot Q_{jj}^{+}(x) \ge (Q^+_{ij}(x) \tplus Q^-_{ij}(x))^{\tdot 2}$;
\item for all $i,j \in [m]$, $i <j$, $(i,j) \in \Sigma^{\complement}$, $Q_{ij}^{+}(x) \mathrel{\Diamond}_{(i,j)} Q_{ij}^{-}(x)$.
\end{itemize}

Observe that we have the equality
\begin{equation}\label[equation]{eq:decomposition}
\spectra = \bigcup_{\Sigma} \bigcap_{\Diamond} \spectra_{\Sigma, \Diamond} \, ,
\end{equation}
where the intersection goes over every $\Diamond \in \{\le, \ge \}^{\Sigma^{\complement}}$ and the union goes over every $\Sigma \subset \{(i,j) \in [m]^{2} \colon i < j\}$. Moreover, we note that every set $\spectra_{\Sigma, \Diamond}$ is a tropical Metzler spectrahedron (we give a detailed proof of this fact in the proof of \cref{corollary:genericity}).

Let $\bo Q^{(1)}, \dots, \bo Q^{(n)} \in \puiseux^{m \times m}$ be any symmetric matrices such that $\sval(\bo Q^{(k)}) = Q^{(k)}$, and $\bspectra \coloneqq \{ \bo x \in \nnpuiseux^n \colon \bo Q(\bo x) \loew 0 \}$ be the associated spectrahedron. We will use the following observation, which already appeared in the proof of \cite[Corollary~3.6]{tropical_simplex} on the tropicalization of polyhedra.
We denote by $\conv(S)$ the convex hull of the set $S \subset \puiseux^n$. 
\begin{lemma}\label[lemma]{lemma:zero_in_convex_hull}
Let $\bo a^{(1)}, \ldots, \bo a^{(p)} \in \puiseux^{n}$ and $\bo b \in \puiseux^{p}$. Suppose that for every sign pattern $\delta \in \{+1, -1\}^{p}$ there is a point $\bo x^{\delta} \in \puiseux^{n}$ such that for all $s \in [p]$ we have $\delta_{s}(\langle \bo a^{(s)}, \bo x^{\delta} \rangle - \bo b_{s}) \geq 0$. Then, there exists a point $\bo y \in \conv_{\delta} \{\bo x^{\delta} \}$ such that for all $s$ we have $\langle \bo a^{(s)}, \bo y \rangle = \bo b_{s}$.
\end{lemma}
\begin{proof}
If $p = 1$ then we have two points $\bo x^{(1)}$ and $\bo x^{(2)}$ such that $\langle \bo a^{(1)}, \bo x^{(1)} \rangle \geq \bo b_{1}$ and $\langle \bo a^{(1)}, \bo x^{(2)} \rangle \leq \bo b_{1}$. Therefore, there exists $\bo \lambda$ such that $0 \leq \bo \lambda \leq 1$ and  $\langle \bo a^{(1)}, \bo \lambda \bo x^{(1)} + (1- \bo \lambda) \bo x^{(2)} \rangle = \bo b_{1}$. This completes the proof for $p = 1$. 

Suppose that the claim is true for $p$. We will prove it for $p + 1$. Take 
\[
\Delta^{+} \coloneqq \{ \delta \in \{+1, -1 \}^{p+1} \colon \text{last entry of $\delta$ is equal to $+1$}\}
\]
and 
\[
\Delta^{-} \coloneqq \{ \delta \in \{+1, -1 \}^{p+1} \colon \text{last entry of $\delta$ is equal to $-1$}\} \, .
\]
By the induction hypothesis, there exists a point $\bo x^{(1)} \in \conv_{\delta \in \Delta^{+}} \{ \bo x^{\delta}\}$ such that  $\langle \bo a^{(s)}, \bo x^{(1)} \rangle = \bo b_{s}$ for all $s \le p$. Moreover, we have $\langle \bo a^{(p+1)}, \bo x^{\delta} \rangle \geq \bo b_{p+1}$ for all $\delta \in \Delta^{+}$ and therefore $\langle \bo a^{(p+1)}, \bo x^{(1)} \rangle \geq \bo b_{p+1}$. Analogously, there exists a point $\bo x^{(2)} \in \conv_{\delta \in \Delta^{-}} \{ \bo x^{\delta}\}$ such that $\langle \bo a^{(s)}, \bo x^{(2)} \rangle = \bo b_{s}$ for all $s \le p$ and $\langle \bo a^{(p+1)}, \bo x^{(2)} \rangle \leq \bo b_{p+1}$. Therefore, there is a point $\bo y \in \conv\{\bo x^{(1)}, \bo x^{(2)}\} \subset \conv_{\delta} \{ \bo x^{\delta}\}$ such that $\langle \bo a^{(p+1)}, \bo y \rangle = \bo b_{p+1}$. Furthermore, since $\langle \bo a^{(s)}, \bo x^{(1)} \rangle = \langle \bo a^{(s)}, \bo x^{(2)} \rangle = \bo b_{s}$ for all $s \le p$, we have $\langle \bo a^{(s)}, \bo y \rangle = \bo b_{s}$ for all $s \le p$.
\end{proof}

\begin{lemma}\label[lemma]{lemma:sout}
We have the inclusion $\val(\outbspectra \cap \pospuiseux^{n}) \subset \spectra \cap \R^{n}$.
\end{lemma}
\begin{proof}
Take a point $\bo x \in \outbspectra \cap \pospuiseux^{n}$ and denote $x \coloneqq \val(\bo x)$. For every $i \in [m]$ we have $\bo Q_{ii}(\bo x) \ge 0$ and hence $Q^{+}_{ii}(x) \ge Q_{ii}^{-}(x)$. Furthermore, for every $i < j$ such that $\val(\bo Q_{ij}^{+}(\bo x)) \neq \val(\bo Q_{ij}^{-}(\bo x))$, we have $\val(\bo Q_{ij}(\bo x)) = Q^+_{ij}(x) \tplus Q^-_{ij}(x)$ and hence $Q^{+}_{ii}(x) \tdot Q^{+}_{jj}(x) \ge (Q^+_{ij}(x) \tplus Q^-_{ij}(x))^{\tdot 2}$. On the other hand, for every $i < j$ such that $\val(\bo Q_{ij}^{+}(\bo x)) = \val(\bo Q_{ij}^{-}(\bo x))$ we have $Q_{ii}^{+}(x) = Q_{jj}^{-}(x)$. In particular, $x \in \spectra \cap \R^{n}$.
\end{proof}

In \cref{lemma:inclusion_of_interior} we introduced the symbol $\interspectra$ to denote the set of all real points that fulfill the strict version of nontrivial inequalities defining a tropical Metzler spectrahedron $\spectra$. Likewise, we denote by $\interspectra_{\Sigma, \Diamond}$ the set of all points $x \in \R^{n}$ which fulfill the strict versions of (nontrivial) inequalities defining $\spectra_{\Sigma, \Diamond}$.

\begin{lemma}\label[lemma]{lemma:sin}
We have
\[
\bigcup_{\Sigma} \bigcap_{\Diamond} \cl\bigl(\interspectra_{\Sigma, \Diamond} \bigr) \subset \val(\inbspectra \cap \pospuiseux^{n}) \, .
\]
\end{lemma}
\begin{proof}
Fix any $\Sigma$ and take $x \in \bigcap_{\Diamond} \cl\bigl(\interspectra_{\Sigma, \Diamond} \bigr)$. By \cref{lemma:inclusion_of_interior}, for every $\Diamond \in \{\le, \ge \}^{\Sigma^{\complement}}$ there exists a lift $\bo x^{\Diamond} \in \pospuiseux^{n} \cap \val^{-1}(x)$ such that we have the inequalities
\begin{align*}
\forall i, \bo Q_{ii}(\bo x^{\Diamond}) &\ge 0 \, , \\
\forall (i,j) \in \Sigma, \bo Q_{ii}(\bo x^{\Diamond})\bo Q_{jj}(\bo x^{\Diamond}) &\ge (m - 1)^{2}(\bo Q_{ij}(\bo x^{\Diamond}))^{2} \, , \\
\forall (i,j) \in \Sigma^{\complement}, \bo Q_{ij}(\bo x^{\Diamond}) & \mathrel{\Diamond}_{(i,j)} 0 \, .
\end{align*}
Observe that the set
\[
\{ \bo y \in \nnpuiseux^{n} \colon \forall i, \bo Q_{ii}(\bo y) \ge 0 \land \forall (i,j) \in \Sigma, \bo Q_{ii}(\bo y)\bo Q_{jj} (\bo y)\ge (m - 1)^{2}(\bo Q_{ij}(\bo y))^{2} \}
\]
is convex. Indeed, it is a spectrahedron defined by some block diagonal matrices with blocks of size at most $2$. Therefore, by \cref{lemma:zero_in_convex_hull}, there exists a point $\bo z \in \conv_{\Diamond}\{\bo x^{\Diamond} \}$ such that
\begin{align*}
\forall i, \bo Q_{ii}(\bo z) &\ge 0 \, , \\
\forall (i,j) \in \Sigma, \bo Q_{ii}(\bo z)\bo Q_{jj}(\bo z) &\ge (m - 1)^{2}(\bo Q_{ij}(\bo z))^{2} \, , \\
\forall (i,j) \in \Sigma^{\complement}, \bo Q_{ij}(\bo z) &= 0 \, .
\end{align*}
In particular, we have $\bo Q_{ii}(\bo z)\bo Q_{jj}(\bo z) \ge (m - 1)^{2}(\bo Q_{ij}(\bo z))^{2}$ for all $(i,j)$ such that $i \neq j$. Therefore $\bo z \in \inbspectra$. Moreover, since $\bo x^{\Diamond} \in \pospuiseux^{n} \cap \val^{-1}(x)$ for all $\Diamond$, we have $\bo z \in \pospuiseux^{n} \cap \val^{-1}(x)$.
\end{proof}
\begin{theorem}\label[theorem]{corollary:genericity}
Let $\bspectra = \{ \bo x \in \nnpuiseux^n \colon \bo Q(\bo x) \loew 0 \}$ be a spectrahedron described by matrices $\bo Q^{(1)}, \dots, \bo Q^{(n)}$ such that $\sval(\bo Q^{(k)}) = Q^{(k)}$. Suppose that \cref{assumption:nondeg_minors} is fulfilled and that every stratum of the set
$\spectra_{\Sigma, \Diamond}(Q^{(1)}, \dots, Q^{(n)})$ is regular for every choice of $(\Sigma, \Diamond)$. Then
\[
\val(\bspectra) = \spectra(Q^{(1)}, \dots, Q^{(n)}) \, .
\]
\end{theorem}

\begin{proof}
We focus on the proof of the identity $\val(\bspectra \cap \pospuiseux^n) = \spectra(Q^{(1)}, \dots, Q^{(n)}) \cap \R^n$, as the generalization to all strata can be obtained analogously to the proof of \cref{corollary:metzler_regularity}. Let $(\Sigma, \Diamond)$ be fixed and observe that $\spectra_{\Sigma,\Diamond}$ is a tropical Metzler spectrahedron. More precisely, it is described by the following tropical block diagonal matrices

\begin{equation}
\begin{bmatrix}
\tilde{Q}^{(k)} & -\infty \\ -\infty & R^{(k)} 
\end{bmatrix} \, ,
\end{equation}
where $\tilde{Q}^{(k)} \in \strop^{m \times m}$ is the symmetric matrix defined by
\[
\tilde{Q}^{(k)}_{ij} \coloneqq
\begin{cases}
Q^{(k)}_{ii} & \text{if} \ i = j \, , \\
\tminus \abs{Q^{(k)}_{ij}} & \text{if} \ (i,j) \in \Sigma \, , \\
-\infty & \text{if} \ (i,j) \in \Sigma^{\complement} \, .
\end{cases}
\]
and $R^{(k)} \in \strop^{\Sigma^{\complement} \times \Sigma^{\complement}}$ is the (tropical) diagonal matrix consisting of the coefficients $Q^{(k)}_{ij}$ if $\Diamond_{(i,j)}$ is equal to $\geq$ and $\tminus Q^{(k)}_{ij}$ otherwise, where $(i,j)$ ranges over the set $\Sigma^{\complement}$.

It can be verified that, as soon as the matrices $Q^{(k)}$ satisfy \cref{assumption:nondeg_minors}, this assumption is also satisfied by all block matrices $\begin{bsmallmatrix} \tilde{Q}^{(k)} & -\infty \\ -\infty & R^{(k)} \end{bsmallmatrix}$. In consequence, as shown in the proof of \cref{th:metzler_positive_orthant}, the sets $\cl(\interspectra_{\Sigma,\Diamond})$ and $\spectra_{\Sigma,\Diamond} \cap \R^n$ coincide. Then, the theorem follows from~\cref{eq:decomposition,lemma:sin,lemma:sout}.
\end{proof}
We now show that, under the regularity conditions discussed above, tropical spectrahedra arise as tropicalizations of minors of size $1$ and $2$ of the linear matrix inequalities that define them. To do so, let $\bspectra = \{ \bo x \in \nnpuiseux^n \colon \bo Q(\bo x) \loew 0 \}$ be a spectrahedron described by matrices $\bo Q^{(1)}, \dots, \bo Q^{(n)}$ as above. For every $i, j \in [m]$ such that $i \neq j$, let
\[
\bo P_{ij} \coloneqq \bo Q_{ii}(\bo x) \bo Q_{jj}(\bo x) - \bo Q_{ij}^{2}(\bo x)
\]
denote the corresponding $2 \times 2$ minor of the matrix pencil $\bo Q(\bo x)$. Furthermore, for every $i \neq j$ let $P_{ij} \coloneqq \tropPol(\bo P_{ij})$ and, for every $i \in [m]$, let $P_{ii} \coloneqq \tropPol(\bo Q_{ii})$. In this way, $P_{ij}$ are the tropicalizations of the $1 \times 1$ and $2 \times 2$ minors of $\bo Q(\bo x)$. Finally, let
\[
\minspectra \coloneqq \{x \in \trop^{n} \colon \forall i, \,  P^{+}_{ii}(x) \ge P^{-}_{ii}(x) \, \land \, \forall i \neq j, \, P^{+}_{ij}(x) \ge P^{-}_{ij}(x) \}
\]
denote the set obtained by this formal tropicalization. 

\begin{theorem}\label[theorem]{concide_minors}
  Suppose that the assumptions of \cref{corollary:genericity} are satisfied. Then, the tropical spectrahedron $\val(\bspectra)$ is determined only by the formal tropicalization of $1\times 1$ and $2\times 2$ minors, in other words
  \[\val(\bspectra) = \minspectra
  \enspace .
  \]
\end{theorem}

We shall deduce this theorem from the following general ``sandwich'' lemma, which does not require any genericity assumption.

\begin{lemma}\label[lemma]{le:formal_trop_minors}
We have $\val(\bspectra) \subset \val(\outbspectra) \subset \minspectra \subset \spectra(Q^{(1)}, \dots, Q^{(n)})$.
\end{lemma}
Let us first explain why \cref{concide_minors}
follows in a straightforward way from \cref{le:formal_trop_minors}.
\begin{proof}[Derivation of \cref{concide_minors} from \cref{le:formal_trop_minors}]
By \cref{le:formal_trop_minors} we have $\val(\bspectra) \subset \minspectra \subset \spectra(Q^{(1)}, \dots, Q^{(n)})$. Furthermore, \cref{corollary:genericity} implies that $\val(\bspectra) = \spectra(Q^{(1)}, \dots, Q^{(n)})$. Hence, we have the equality $\val(\bspectra) = \minspectra$.
\end{proof}
The proof of \cref{le:formal_trop_minors} requires more technical efforts.

\begin{proof}[Proof of~\cref{le:formal_trop_minors}]
  The inclusion $\val(\bspectra) \subset \val(\outbspectra)$ is trivial. The inclusion $\val(\outbspectra) \subset \minspectra$ follows from \cref{lemma:homomorphism}, as in the proof of \cref{corollary:purity_commutation}. It remains to show that $\minspectra$ is included in $\spectra(Q^{(1)}, \dots, Q^{(n)})$. To do so,
  let us choose an arbitrary point $\barx \in \minspectra$.
  We first observe that $P_{ii} = Q_{ii}$ for all $i$ by the definition of the formal tropicalization. In particular,
  $Q_{ii}^{+}(\barx) \ge Q_{ii}^{-}(\barx)$ for all $i \in [m]$.
  
  We must
  show that
  \begin{align*} Q_{ii}^{+}(\barx) \tdot Q_{jj}^{+}(\barx) \ge (Q^+_{ij}(\barx) \tplus Q^-_{ij}(\barx))^{\tdot 2}
  \text{ or }Q_{ij}^{+}(\barx) = Q_{ij}^{-}(\barx)\text{ for all }i \neq j
  \enspace .
  \end{align*}
  It suffices to consider $i \neq j$ such that $Q_{ij}^{+}(\barx) \neq Q^{-}_{ij}(\barx)$ and to show that $Q_{ii}^{+}(\barx) \tdot Q_{jj}^{+}(\barx) \ge (Q^+_{ij}(\barx) \tplus Q^-_{ij}(\barx))^{\tdot 2}$.
  Since $Q_{ij}^{+}(\barx) \neq Q^{-}_{ij}(\barx)$, at least one of the polynomials $Q_{ij}^{+}, Q^{-}_{ij}$ is nonzero. By the definition of $Q_{ij}^{+}$ and $Q^{-}_{ij}$ we have $Q_{ij}^{+}(x) \tplus Q^{-}_{ij}(x) = \max_{k \in [n]}\{\val(\bo Q_{ij}^{(k)}) + x_{k}\}$. Let $k' \in [n]$ be an index
such that
\begin{equation}\label{eq:max_index_minor}
Q_{ij}^{+}(\barx) \tplus Q^{-}_{ij}(\barx) = \val(\bo Q_{ij}^{(k')}) + \barx_{k'} \, .
\end{equation}
Let us point out that $\val(\bo Q_{ij}^{(k')}) \neq -\infty$ and $\barx_{k'} \neq -\infty$. Indeed, $Q_{ij}^{+}(\barx) \neq Q^{-}_{ij}(\barx)$ implies that $Q_{ij}^{+}(\barx) \tplus Q^{-}_{ij}(\barx) \neq -\infty$, so neither $\val(\bo Q_{ij}^{(k')})$ nor $\barx_{k'}$ can be equal to $-\infty$.
To prove the claim, it remains to show that
\begin{align}
  2\val(\bo Q_{ij}^{(k')}) + 2\barx_{k'} \le Q_{ii}^{+}(\barx) \tdot Q_{jj}^{+}(\barx)
  \enspace .\label{e-intermediate}
  \end{align}
Consider two cases. First, suppose that $ 2\val(\bo Q_{ij}^{(k')}) \le \val(\bo Q_{ii}^{(k')}\bo Q_{jj}^{(k')})$. In this case, using the inequalities $Q^{+}_{ii}(\barx) \ge  Q^{-}_{ii}(\barx)$ and $Q^{+}_{jj}(\barx) \ge  Q^{-}_{jj}(\barx)$, we obtain
\begin{align*}
&2\val(\bo Q_{ij}^{(k')}) + 2\barx_{k'} \le \val(\bo Q_{ii}^{(k')}\bo Q_{jj}^{(k')}) + 2\barx_{k'} \\
&\le \bigl(Q^{+}_{ii}(\barx) \tplus  Q^{-}_{ii}(\barx)\bigr) \tdot \bigl(Q^{+}_{jj}(\barx) \tplus  Q^{-}_{jj}(\barx)\bigr) =  Q^{+}_{ii}(\barx) \tdot  Q^{+}_{jj}(\barx) \, ,
\end{align*}
and so,~\cref{e-intermediate} is established in this case.
Second, suppose that $ 2\val(\bo Q_{ij}^{(k')}) > \val(\bo Q_{ii}^{(k')}\bo Q_{jj}^{(k')})$.  This case requires a longer argument.

To start, note that $ 2\val(\bo Q_{ij}^{(k')}) > \val(\bo Q_{ii}^{(k')}\bo Q_{jj}^{(k')})$ implies that $(\bo Q_{ij}^{(k')})^{2} > \bo Q_{ii}^{(k')}\bo Q_{jj}^{(k')}$ and that $\val\bigl(\bo Q_{ii}^{(k')}\bo Q_{jj}^{(k')} - (\bo Q_{ij}^{(k')})^{2} \bigr) = 2\val(\bo Q_{ij}^{(k')})$. Furthermore, the minor $\bo P_{ij}$ is given as
\begin{equation}
\begin{aligned}
\bo P_{ij}(\bo x) = &\sum_{k \in [n]} \bigl( \bo Q^{(k)}_{ii}\bo Q^{(k)}_{jj} - (\bo Q^{(k)}_{ij})^{2} \bigr) \bo x_{k}^{2} \\
&+ \sum_{k < l}(\bo Q_{ii}^{(k)}\bo Q_{jj}^{(l)} + \bo Q_{ii}^{(l)}\bo Q_{jj}^{(k)} - 2\bo Q_{ij}^{(k)}\bo Q_{ij}^{(l)})\bo x_{k} \bo x_{l} \, .\label{e-temp}
\end{aligned}
\end{equation}
In particular, the polynomial $\bo P_{ij}(\bo x)$ has at least one monomial with (strictly) negative coefficient, namely $(\bo Q_{ii}^{(k')}\bo Q_{jj}^{(k')} - (\bo Q_{ij}^{(k')})^{2}) \bo x_{k'}^{2}$.
Remembering that $P^-_{ij}=\tropPol(\bm P^-_{ij})$ is the tropical polynomial whose monomials are obtained by tropicalizing the monomials of $\bm P_{ij}$ with negative coefficients (as in \cref{section:tropical_algebra}), we see that $P^{-}_{ij}$ is a nonzero tropical polynomial and that it contains a monomial of the form $\val(\bo Q_{ij}^{(k')})^{\tdot 2} \tdot x_{k'}^{\tdot 2}$. Furthermore, recall that we have $P^{+}_{ij}(\barx) \ge P^{-}_{ij}(\barx)$
and that we already showed that $\barx_{k'} \neq -\infty$. Hence, we also have $P^{-}_{ij}(\barx) \neq -\infty$, which implies that $P^{+}_{ij}(\barx) \neq -\infty$. Therefore, the polynomial $\bo P_{ij}$ has at least one monomial with (strictly) positive coefficient. In other words, there exist indices $k, l \in [n]$ (possibly equal) such that 
\[
\bo Q_{ii}^{(k)}\bo Q_{jj}^{(l)} + \bo Q_{ii}^{(l)}\bo Q_{jj}^{(k)} - 2\bo Q_{ij}^{(k)}\bo Q_{ij}^{(l)} > 0 \, .
\]
By definition, $P^+_{ij}(\barx)$
is the maximum of the tropical monomials arising from the monomials
with positive coefficients in~\cref{e-temp}, evaluated at point $\barx$.
Choosing a monomial attaining this maximum yields a pair of indices $(k,l)$ such that $P^{+}_{ij}(\barx) = \val(\bo Q_{ii}^{(k)}\bo Q_{jj}^{(l)} + \bo Q_{ii}^{(l)}\bo Q_{jj}^{(k)} - 2\bo Q_{ij}^{(k)}\bo Q_{ij}^{(l)}) + \barx_{k} + \barx_{l}$. Since $P^{+}_{ij}(\barx) \ge P^{-}_{ij}(\barx)$, we also have

\[
P^{+}_{ij}(\barx) \tplus P^{-}_{ij}(\barx) = \val(\bo Q_{ii}^{(k)}\bo Q_{jj}^{(l)} + \bo Q_{ii}^{(l)}\bo Q_{jj}^{(k)} - 2\bo Q_{ij}^{(k)}\bo Q_{ij}^{(l)}) + \barx_{k} + \barx_{l} \, .
\]
Hence, we get
\begin{equation}\label{eq:compute_minors}
\begin{aligned}
&2\val(\bo Q_{ij}^{(k')}) + 2\barx_{k'} = \val\bigl(\bo Q_{ii}^{(k')}\bo Q_{jj}^{(k')} - (\bo Q_{ij}^{(k')})^{2} \bigr) + 2\barx_{k'} \\
&\le P^{+}_{ij}(\barx) \tplus P^{-}_{ij}(\barx) \le \max\{\val(\bo Q_{ii}^{(k)}\bo Q_{jj}^{(l)} + \bo Q_{ii}^{(l)}\bo Q_{jj}^{(k)}), \val(\bo Q_{ij}^{(k)}\bo Q_{ij}^{(l)})\} + \barx_{k} + \barx_{l} \, .
\end{aligned}
\end{equation}
We will show that 
\[
\max\{\val(\bo Q_{ii}^{(k)}\bo Q_{jj}^{(l)} + \bo Q_{ii}^{(l)}\bo Q_{jj}^{(k)}), \val(\bo Q_{ij}^{(k)}\bo Q_{ij}^{(l)})\} = \val(\bo Q_{ii}^{(k)}\bo Q_{jj}^{(l)} + \bo Q_{ii}^{(l)}\bo Q_{jj}^{(k)}) \, .
\]
To do so, suppose that $ \val(\bo Q_{ij}^{(k)}\bo Q_{ij}^{(l)}) > \val(\bo Q_{ii}^{(k)}\bo Q_{jj}^{(l)} +  \bo Q_{ii}^{(l)}\bo Q_{jj}^{(k)})$. Then, we have $\bo Q_{ij}^{(k)}\bo Q_{ij}^{(l)} < 0$, because $\bo Q_{ii}^{(k)}\bo Q_{jj}^{(l)} + \bo Q_{ii}^{(l)}\bo Q_{jj}^{(k)}- 2\bo Q_{ij}^{(k)}\bo Q_{ij}^{(l)} \ge 0$. In particular, since $Q_{ij}^{+}(\barx) \neq Q^{-}_{ij}(\barx)$, we obtain
\begin{align*}
2\val(\bo Q_{ij}^{(k')}) + 2\barx_{k'} &\le \val(\bo Q_{ij}^{(k)}\bo Q_{ij}^{(l)}) + \barx_{k} + \barx_{l} \\
&\le Q_{ij}^{+}(\barx) \tdot Q^{-}_{ij}(\barx) < (Q^+_{ij}(\barx) \tplus Q^-_{ij}(\barx))^{\tdot 2} \, ,
\end{align*}
which gives a contradiction with~\cref{eq:max_index_minor}. Hence, $\val(\bo Q_{ii}^{(k)}\bo Q_{jj}^{(l)} + \bo Q_{ii}^{(l)}\bo Q_{jj}^{(k)}) \ge  \val(\bo Q_{ij}^{(k)}\bo Q_{ij}^{(l)})$ and~\cref{eq:compute_minors} gives
\begin{equation}
\begin{aligned}
2\val(\bo Q_{ij}^{(k')}) + 2\barx_{k'}  &\le \val(\bo Q_{ii}^{(k)}\bo Q_{jj}^{(l)} + \bo Q_{ii}^{(l)}\bo Q_{jj}^{(k)}) + \barx_{k} + \barx_{l} \\
&\le \max\{\val(\bo Q_{ii}^{(k)}\bo Q_{jj}^{(l)}) + \barx_{k} + \barx_{l}, \val(\bo Q_{ii}^{(l)}\bo Q_{jj}^{(k)}) + \barx_{k} + \barx_{l}\}\\
&\le (Q^{+}_{ii}(\barx) \tplus Q^{-}_{ii}(\barx)) \tdot  (Q^{+}_{jj}(\barx) \tplus Q^{-}_{jj}(\barx)) \\
&= Q^{+}_{ii}(\barx) \tdot Q^{+}_{jj}(\barx) \, ,
\end{aligned}
\end{equation}
which shows~\cref{e-intermediate}. This concludes the proof of \cref{le:formal_trop_minors}, and so, of \cref{concide_minors}.
\end{proof}

\begin{figure}[t]
\centering
\begin{tikzpicture}
\begin{scope}[scale = 0.6]
      \draw[gray!80!black, thin, ->] (0,4) -- (8,4) node[below right] {$x_1$};
      \draw[gray!90!black, thin, ->] (4,0) -- (4,8) node[below left] {$x_2$};
      \draw[gray!60, ultra thin] (0,0) grid (8,8);
       \fill[fill=lightgray, fill opacity = 0.7]
        (0,4) -- (4,4) -- (4,0) -- (0,0) -- cycle;
        \draw[very thick] (0,4) -- (4,4) -- (4,0);
        \draw[very thick] (4,4) -- (8,8);
\end{scope}
\end{tikzpicture}
\caption{A nonregular tropical spectrahedron that fulfills the regularity conditions of \cref{corollary:genericity}.}\label[figure]{fig:spectra_not_pure}
\end{figure}
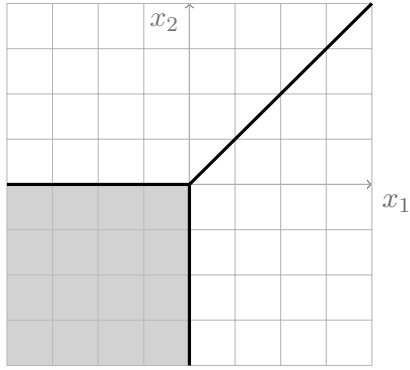

\begin{example}\label[example]{ex:spectra_not_pure}
Take the matrices
\[ 
Q^{(0)} \coloneqq \begin{bmatrix} a & \zero \\ \zero & b \end{bmatrix}, \ Q^{(1)} \coloneqq \begin{bmatrix} \zero & c \\ c & \zero \end{bmatrix}, \ Q^{(2)} \coloneqq \begin{bmatrix} \zero & \tminus d \\ \tminus d & \zero \end{bmatrix} \, .
\]
The set $\spectra(Q^{(0)}, Q^{(1)}, Q^{(2)})$ is given by
\[
\{x \in \trop^{3} \colon (a + x_{0}) + (b + x_{0}) \ge 2\max\{c+x_{1}, d + x_{2}\} \text{ or } c + x_{1} = d + x_{2} \} \, .
\]
Moreover, $\spectra(Q^{(0)}, Q^{(1)}, Q^{(2)})$ fulfills the conditions of \cref{corollary:genericity}. \Cref{fig:spectra_not_pure} depicts the intersection of this tropical spectrahedron with the hyperplane $\{x_{0} = 0 \}$ for $a=b=c=d=0$. Note that this tropical spectrahedron is not regular for any choice of $a,b,c,d \in \R$. Furthermore, let
\[
\bo Q^{(0)} \coloneqq \begin{bmatrix} t^{a} & 0 \\ 0 & t^{b} \end{bmatrix}, \ \bo Q^{(1)} \coloneqq \begin{bmatrix} 0 & t^{c} \\ t^{c} & 0 \end{bmatrix}, \ \bo Q^{(2)} \coloneqq \begin{bmatrix} 0 & -t^{d} \\ -t^{d} & 0 \end{bmatrix} \, 
\]
be a lift of the matrices $\bo Q^{(k)}$. Then, the associated tropical spectrahedron $\bspectra$ is given by
\[
\bspectra = \{\bo x \in \nnpuiseux^{3} \colon t^{a+b}\bo x_{0}^2 -t^{2c}\bo x_{1}^{2} - t^{2d}\bo x_{2}^{2} + 2t^{c + d}\bo x_{1} \bo x_{2} \ge 0 \} \, .
\]
The formal tropicalization $\minspectra$ of this set is given by
\[
\{x \in \trop^{3} \colon ((a+b)\tdot x_{0}^{\tdot 2}) \tplus ((c+d) \tdot x_{1} \tdot x_{2}) \ge (2c \tdot x_{1}^{\tdot 2}) \tplus (2d \tdot x_{2}^{\tdot 2}) \} \, .
\]
It is easy to check that $\minspectra$ coincides with $\spectra(Q^{(0)}, Q^{(1)}, Q^{(2)})$, as implied by \cref{concide_minors}.
\end{example}

\subsection{Genericity conditions}\label[section]{section:generic}
In this section we show that the requirements of \cref{th:metzler_positive_orthant,corollary:genericity} on the matrices $Q^{(k)}$ and the regularity of sets are fulfilled generically. In \cite{tropical_simplex} it was shown that genericity conditions for tropical polyhedra can be described by the means of tangent digraphs. We extend this characterization to tropical spectrahedra. For this purpose, we work with hypergraphs instead of graphs. A \emph{(directed) hypergraph} is a pair $\dgraph \coloneqq (\vertices, \edges)$, where $\vertices$ is a finite set of \emph{vertices} and $\edges$ is a finite set of \emph{(hyper)edges}. Every edge $\edge \in \edges$ is a pair $(\tails_{\edge}, \head_{\edge})$, where $\head_{\edge} \in \vertices$ is called the \emph{head} of the edge, and $\tails_{\edge}$ is a multiset with elements taken from $\vertices$. We call $\tails_{\edge}$ the \emph{multiset of tails of $\edge$}. By $\card{\tails_{\edge}}$ we denote the cardinality of $\tails_{\edge}$ (counting multiplicities). Note that we do not exclude the situation in which a head is also a tail, i.e., it is possible that $\head_{\edge} \in \tails_{\edge}$.

Let us now define the notion of a circulation in a hypergraph. If $\vertex \in \vertices$ is a vertex, then by $\inedge(\vertex) \subset \edges$ we denote the set of \emph{incoming edges}, i.e., the set of all edges $e$ such that $\head_{\edge} = \vertex$. By $\outedge(\vertex)$ we denote the multiset of \emph{outgoing edges}, i.e., a multiset of edges $\edge$ such that $\vertex \in \tails_{\edge}$. We treat $\outedge(\vertex)$ as a multiset, with the convention that $\edge \in \edges$ appears $p$ times in $\outedge(\vertex)$ if $\vertex$ appears 
$p$
times in $\tails_{\edge}$. A \emph{circulation} in a hypergraph is a nonzero vector $\circulation = (\circulation_{\edge})_{\edge \in \edges}$ such that $\circulation_{\edge} \ge 0$ for all $\edge \in \edges$, and for all $\vertex \in \vertices$ we have the equality
\[
\sum_{\edge \in \inedge(\vertex)} \card{\tails_{\edge}}\circulation_{\edge} = \sum_{\edge \in \outedge(\vertex)} \circulation_{\edge} \, .
\]
We always suppose that circulations are normalized, i.e., that $\sum_{\edge \in \edges} \circulation_{\edge} = 1$. Observe that if a hypergraph $\dgraph$ is fixed, then the set of all normalized circulations on $\dgraph$ forms a polytope. We say that a hypergraph \emph{does not admit a circulation} if this polytope is empty.

In our framework, every edge has at most two tails (counting multiplicities).  Hereafter, $\stbase_{k}$ denotes the $k$th vector of the standard basis in $\R^{n}$. Given a sequence of tropical symmetric Metzler matrices $Q^{(1)}, \dots, Q^{(n)} \in \strop^{m \times m}$, and a point $x \in \R^{n}$, we construct a hypergraph associated with $x$, denoted $\dgraph_{x}$, as follows:
\begin{itemize}
\item we put $\vertices \coloneqq [n]$;
\item for every $i \in [m]$ verifying $Q_{ii}^{+}(x) = Q_{ii}^{-}(x) \neq \zero$, and every pair 
$\stbase_{k} \in \domin(Q_{ii}^{+}, x)$, $\stbase_{l} \in \domin(Q_{ii}^{-}, x)$, $\dgraph_{x}$ contains an edge $(k,l)$;
\item for every $i < j$ such that $Q_{ii}^{+}(x) \tdot Q_{jj}^{+}(x) = (Q_{ij}(x))^{\tdot 2} \neq \zero$ and every triple $\stbase_{k_{1}} \in \domin(Q_{ii}^{+}, x)$, $\stbase_{k_{2}} \in \domin(Q_{jj}^{+}, x)$, $\stbase_{l} \in \domin(Q_{ij}, x)$, $\dgraph_{x}$ contains an edge $(\{k_{1}, k_{2}\},l)$.
\end{itemize}
\begin{lemma}\label[lemma]{lemma:circulation}
Suppose that for every $x \in \R^{n}$ the hypergraph $\dgraph_{x}$ does not admit a circulation. Then the matrices $Q^{(1)}, \dots, Q^{(n)}$ fulfill \cref{assumption:nondeg_minors} and $\spectra(Q^{(1)}, \dots, Q^{(n)}) \cap \R^{n}$ is regular.
\end{lemma}
\begin{proof}
To prove the first part, suppose that we have $Q^{(k)}_{ii} + Q^{(k)}_{jj} = 2\abs{Q^{(k)}_{ij}}$ for some $i \neq j$ and $Q^{(k)}_{ii}, Q^{(k)}_{jj} \in \postrop$. Take the point $x \coloneqq N\stbase_{k} \in \R^{n}$. If $N$ is large enough, then we have 
\[
Q_{ii}^{+}(x) \tdot Q_{jj}^{+}(x) = Q^{(k)}_{ii} + Q^{(k)}_{jj} + 2N =  2\abs{Q^{(k)}_{ij}} + 2N = (Q_{ij}(x))^{\tdot 2} \,
\]
and the hypergraph $\dgraph_{x}$ contains the edge $(\{k,k\}, k)$. This hypergraph admits a circulation (we put $\circulation_{\edge} \coloneqq 1$ for $\edge = (\{k,k\}, k)$ and $\circulation_{\edge} \coloneqq 0$ for other edges).

We now claim that the set $\spectra \cap \R^{n}$ is regular. Let $\interspectra$ be defined as in \cref{lemma:inclusion_of_interior}. Let us show that for every $x \in \spectra \cap \R^{n}$ there exists a vector $\eta \in \R^{n}$ such that $x + \rho \eta$  belongs to $\interspectra$ for $\rho > 0$ small enough. This is sufficient to prove the claim because $\interspectra$ is a subset of the interior of $\spectra \cap \R^{n}$. Fix a point $x \in \spectra \cap \R^{n}$. If $x$ belongs to $\interspectra$, then we can take $\eta \coloneqq 0$. Otherwise, let $\dgraph_{x}$ denote the hypergraph associated with $x$. The polytope of normalized circulations of this hypergraph is empty. Therefore, by Farkas' lemma, there exists a vector $\eta \in \R^{n}$ such that for every edge $\edge \in \edges$ we have
\[
\sum_{\vertex \in \tails_{\edge}} \eta_{\vertex} > \card{\tails_{\edge}} \eta_{\head_{\edge}} \, .
\]
Take the vector $x^{(\rho)} \coloneqq x + \rho \eta$. Let us look at two cases.

First, suppose that there is $i \in [m]$ such that $Q^{+}_{ii}(x) = Q^{-}_{ii}(x) \neq \zero$. Fix any $k^{*}$ such that $\stbase_{k^{*}} \in \domin(Q_{ii}^{+}, x)$ and take any $l$ such that $\stbase_{l} \in \domin(Q_{ii}^{-}, x)$. Then $(k^{*},l)$ is an edge in $\dgraph_{x}$. Therefore $\eta_{k^{*}} > \eta_{l}$. Moreover, $Q_{ii}^{(k^{*})} + x_{k^{*}} = \abs{Q_{ii}^{(l)}} + x_{l}$ and hence $Q_{ii}^{(k^{*})} + x^{(\rho)}_{k^{*}} > \abs{Q_{ii}^{(l)}} + x^{(\rho)}_{l}$. Furthermore, for every $l' \notin \domin(Q_{ii}^{-}, x)$ we have
\[
Q_{ii}^{(k^{*})} + x_{k^{*}} = \abs{Q_{ii}^{(l)}} + x_{l} > \abs{Q_{ii}^{(l')}} + x_{l'} \, .
\]
Therefore $Q_{ii}^{(k^{*})} + x^{(\rho)}_{k^{*}} > \abs{Q_{ii}^{(l')}} + x^{(\rho)}_{l'}$ for $\rho$ small enough. Since $l$, $l'$ were arbitrary, for every sufficiently small $\rho$ we have
\[
Q_{ii}^{+}(x^{(\rho)}) \ge Q_{ii}^{(k^{*})} + x^{(\rho)}_{k^{*}} > Q_{ii}^{-}(x^{(\rho)}) \, .
\]

The second case is analogous. If there is $i < j$ such that $Q_{ii}^{+}(x) \tdot Q_{jj}^{+}(x) = (Q_{ij}(x))^{\tdot 2} \neq \zero$, then we fix $(k_{1}^{*}, k_{2}^{*})$ such that $\stbase_{k^{*}_{1}} \in \domin(Q_{ii}^{+}, x)$, $\stbase_{k^{*}_{2}} \in \domin(Q_{jj}^{+}, x)$. For every $\stbase_{l} \in \domin(Q_{ij}, x)$, $(\{k^{*}_{1}, k^{*}_{2} \}, l)$ is an edge in $\dgraph$. Hence $\eta_{k^{*}_{1}} + \eta_{k^{*}_{2}} > 2\eta_{l}$. Therefore $Q^{(k^{*}_{1})}_{ii} + Q^{(k^{*}_{2})}_{jj} + x^{(\rho)}_{k^{*}_{1}} + x^{(\rho)}_{k^{*}_{2}} > 2\abs{Q^{(l)}_{ij}} + 2x^{(\rho)}_{l}$. As before, this implies that $Q_{ii}^{+}(x^{(\rho)}) \tdot Q_{jj}^{+}(x^{(\rho)}) > (Q_{ij}(x^{(\rho)}))^{\tdot 2}$ for $\rho > 0$ small enough. Since we supposed that $x \in \spectra \cap \R^{n}$, we have $x^{(\rho)} \in \interspectra$ for $\rho$ small enough.
\end{proof}

\begin{figure}[t]
\centering
\includegraphics[height=2.4cm]{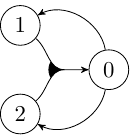}
\caption{The hypergraph from \cref{ex:hypergraph}, consisting of the edges $(0,1)$, $(0,2)$, and $(\{1,2\}, 0)$.}\label[figure]{fig:hypergraph}
\end{figure}

\begin{example}\label[example]{ex:hypergraph}
Take the matrices
\begin{equation*}
\begin{aligned}
Q^{(0)} &\coloneqq \diag \Bigl( 0, 0, \begin{bmatrix} -1 & 0 \\ 0 & -1 \end{bmatrix} \Bigr) \, , \\
Q^{(1)} &\coloneqq \diag \Bigl( \tminus 0,  \zero, 0, \zero \Bigr) \, , \\
Q^{(2)} &\coloneqq \diag \Bigl( \zero,  \tminus 0, \zero, 0 \Bigr) \, .
\end{aligned}
\end{equation*}
In this case, the set $\spectra(Q^{(0)}, Q^{(1)}, Q^{(2)}) \cap \R^3 = \{(\lambda, \lambda, \lambda) \colon \lambda \in \R \}$ is reduced to a line.  In particular, it is not regular. The hypergraph associated with $(0,0,0)$ is depicted in \cref{fig:hypergraph}. This hypergraph admits a circulation (we put $\circulation_{\edge} \coloneqq 1/3$ for all edges).
\end{example}

We now want to show that the condition of \cref{lemma:circulation} is fulfilled generically.
\begin{lemma}\label[lemma]{lemma:metzler_genericity}
There exists a set $X \subset \trop^{d}$ with $d = nm(m+1)/2$ such that every stratum of $X$ is a finite union of hyperplanes and such that if the vector with entries $\abs{Q^{(k)}_{ij}}$ (for $i \le j$) does not belong to $X$, then the hypergraph $\dgraph_{x}$ does not admit a circulation for any $x \in \R^{n}$.
\end{lemma}
\begin{proof}
Fix a nonempty subset $D \subset [d]$, $\card{D} = d'$ and let $\R^{d'}$ be the stratum of $\trop^{d}$ associated with $D$. Suppose that $Q^{(1)}, \dots, Q^{(n)}$ are tropical Metzler matrices, that the support of the vector $\abs{Q^{(k)}_{ij}}$ is equal to $D$, and that $x \in \R^{n}$ is such that $\dgraph_{x}$ admits a circulation. Fix any such circulation $\circulation$. For every edge $\edge = (k,l)$ of $\dgraph$ we can fix $i_{\edge} \in [m]$ such that $Q^{(k)}_{i_{\edge}i_{\edge}} + x_{k} =  \abs{Q^{(l)}_{i_{\edge}i_{\edge}}}  +x_{l} \neq -\infty$. Similarly, for every edge $\edge = (\{k_{1}, k_{2}\},l)$ of $\dgraph$ we can fix $i_{\edge} < j_{\edge}$ such that $Q^{(k_{1})}_{i_{\edge}i_{\edge}} + Q^{(k_{2})}_{j_{\edge}j_{\edge}} + x_{k_{1}} + x_{k_{2}} = 2\abs{Q^{(l)}_{i_{\edge}j_{\edge}}}  + 2x_{l} \neq -\infty$. We take the sum of these equalities weighted by $\circulation$. This gives the equality
\begin{align*}
\sum_{k \in [n]} \sum_{\edge \in \outedge(k)} \circulation_{\edge} Q^{(k)}_{i_{\edge}i_{\edge}} &+ \sum_{k \in [n]} \sum_{\edge \in \outedge(k)} \circulation_{\edge} x_{k} \\ 
= \sum_{l \in [n]} \sum_{\edge \in \inedge_{1}(l)} \circulation_{\edge} \abs{Q^{(l)}_{i_{\edge}i_{\edge}}} &+ \sum_{l \in [n]} \sum_{\edge \in \inedge_{2}(l)} 2\circulation_{\edge} \abs{Q^{(l)}_{i_{\edge}j_{\edge}}} + \sum_{l \in [n]} \sum_{\edge \in \inedge(l)} \card{\tails_{\edge}} \circulation_{\edge} x_{l} \, ,
\end{align*}
where $\inedge_{1}(l)$ denotes the set of incoming edges with tails of cardinality $1$ and $\inedge_{2}(l)$ denotes the set of incoming edges with tails of cardinality $2$. Since $\circulation$ is a circulation, this expression simplifies to
\[
\sum_{k \in [n]} \sum_{\edge \in \outedge(k)} \circulation_{\edge} Q^{(k)}_{i_{\edge}i_{\edge}} 
= \sum_{l \in [n]} \sum_{\edge \in \inedge_{1}(l)} \circulation_{\edge} \abs{Q^{(l)}_{i_{\edge}i_{\edge}}} + \sum_{l \in [n]} \sum_{\edge \in \inedge_{2}(l)} 2\circulation_{\edge} \abs{Q^{(l)}_{i_{\edge}j_{\edge}}} \, .
\]
Consider the set $\hplane$ of all $z \in \R^{d'}$ such that
\begin{equation}
\begin{aligned}
\sum_{k \in [n]} \sum_{\edge \in \outedge(k)} \circulation_{\edge} z^{(k)}_{i_{\edge}i_{\edge}} 
= \sum_{l \in [n]} \sum_{\edge \in \inedge_{1}(l)} \circulation_{\edge} z^{(l)}_{i_{\edge}i_{\edge}} + \sum_{l \in [n]} \sum_{\edge \in \inedge_{2}(l)} 2\circulation_{\edge} z^{(l)}_{i_{\edge}j_{\edge}} 
\, . \label[equation]{eq:hyperplane}
\end{aligned}
\end{equation}
This set is a hyperplane. Indeed, suppose that the equality above is trivial (i.e., that it reduces to $0 = 0$). Take any edge $\edge$ such that $\circulation_{\edge} \neq 0$ and any vertex $k \in \tails_{\edge}$. Then the coefficient $z^{(k)}_{i_{\edge}i_{\edge}}$ appears on the left-hand side. Moreover, we have $\sign(Q^{(k)}_{i_{\edge}i_{\edge}}) = 1$. On the other hand, for every coefficient $z^{(l)}_{j_{\edge}j_{\edge}}$ that appears on the right-hand side we have $\sign(Q^{(l)}_{j_{\edge}j_{\edge}}) = -1$. This gives a contradiction. 

Therefore, we can construct the stratum of $X$ associated with $D$ (denoted $X_{D}$) as follows: we take all possible hypergraphs that can arise in our construction (since $n$ is fixed, we have finitely many of them). Out of them, we choose those hypergraphs that admit a circulation. For every such hypergraph we pick exactly one circulation $\circulation$. After that, for every possible choice of functions $\edge \to i_{\edge}$, $\edge \to (i_{\edge}, j_{\edge})$\footnote{Note that the dependence on $D$ lies here, as the choice of $D$ restricts the amount of possible functions $\edge \to i_{\edge}$, $\edge \to (i_{\edge}, j_{\edge})$.}, we take a set $\hplane$ defined as in \cref{eq:hyperplane}. If $\hplane$ is equal to $\R^{d'}$, then we ignore it. Otherwise, $\hplane$ is a hyperplane. We take $X_{D}$ to be the union of all hyperplanes obtained in this way.
\end{proof}

The proof of \cref{lemma:metzler_genericity} can be easily adapted to give a genericity condition both for Metzler and non-Metzler spectrahedra.

\begin{theorem}\label[theorem]{theorem:genericity}
Let $Q^{(1)}, \dots, Q^{(n)} \in \strop^{m \times m}$ be a sequence of symmetric tropical matrices. There exists a set $X \subset \trop^{d}$ with $d = nm(m+1)/2$ such that every stratum of $X$ is a finite union of hyperplanes and such that if the vector with entries $\abs{Q^{(k)}_{ij}}$ (for $i \le j$) does not belong to $X$, then the matrices $Q^{(1)}, \dots, Q^{(n)}$ fulfill \cref{assumption:nondeg_minors} and for all $(\Sigma, \Diamond)$, every stratum of $\spectra_{\Sigma, \Diamond}(Q^{(1)}, \dots, Q^{(n)})$ is regular.
\end{theorem}

\begin{proof}
As previously, we fix a nonempty set $D \subset [d]$, $\card{D} = d'$, and we will present a construction of the stratum of $X$ associated with $D$, denoted $X_{D}$. Take symmetric matrices $(Q^{(k)}) \in\trop^{m \times m}$ such that the sequence $(\abs{Q^{(k)}_{ij}}) \in \trop^{d}$ has support equal to $D$. Take any nonempty subset $K \subset [n]$ and let $\spectra^{(K)}$ denote the set $\spectra((Q^{(k)})_{k \in K})$. Fix a pair $(\Sigma, \Diamond)$ and take the tropical Metzler spectrahedron $\spectra^{(K)}_{\Sigma, \Diamond}$. Take any $x \in \R^{K}$ and a graph $\dgraph_{x}$ associated with $\spectra^{(K)}_{\Sigma, \Diamond}$ (note that this graph has vertices enumerated by numbers from $K$). Suppose that this graph admits a circulation $\circulation$. As previously, for every edge $\edge = (\{k_{1}, k_{2}\},l)$ of $\dgraph$ we can take $(i_{\edge}, j_{\edge}) \in \Sigma$ such that $Q^{(k_{1})}_{i_{\edge}i_{\edge}} + Q^{(k_{2})}_{j_{\edge}j_{\edge}} + x_{k_{1}} + x_{k_{2}} = 2\abs{Q^{(l)}_{i_{\edge}j_{\edge}}}  + 2x_{l}$. For every edge $\edge = (k,l)$ we have two possibilities: either there exists $i_{\edge} \in [m]$ such that $Q^{(k)}_{i_{\edge}i_{\edge}} + x_{k} =  \abs{Q^{(l)}_{i_{\edge}i_{\edge}}}  +x_{l}$ or there exists $(i_{\edge}, j_{\edge}) \in \Sigma^{\complement}$ such that $\abs{Q^{(k)}_{i_{\edge}j_{\edge}}} + x_{k} =  \abs{Q^{(l)}_{i_{\edge}j_{\edge}}}  +x_{l}$. As before, we take the sum of these equalities weighted by $\circulation$. This gives the identity
\[
\sum_{k \in [n]} \sum_{\edge \in \outedge(k)} \circulation_{\edge} \abs{Q^{(k)}_{i_{\edge}j_{\edge}}} 
= \sum_{l \in [n]} \sum_{\edge \in \inedge_{1}(l)} \circulation_{\edge} \abs{Q^{(l)}_{i_{\edge}j_{\edge}}} + \sum_{l \in [n]} \sum_{\edge \in \inedge_{2}(l)} 2\circulation_{\edge} \abs{Q^{(l)}_{i_{\edge}j_{\edge}}} \, .
\]
As previously, the set of all $z \in \R^{d'}$ that fulfills this equality is a hyperplane. Indeed, any coefficient $z^{(k)}_{i_{\edge}j_{\edge}}$ which appears on the left-hand side does not appear on the right-hand side (note that here we use the fact that $\Sigma \cap \Sigma^{\complement} = \emptyset$). As before, we take all possible hypergraphs (where ``all possible'' takes into account the fact that $K$ can vary), one circulation for each hypergraph, all possible functions $\edge \to (i_{\edge}, j_{\edge})$ (the number of such functions depends on $D$), and all hyperplanes that can arise in this way. The union of these hyperplanes constitutes $X_{D}$. We deduce the result from \cref{lemma:circulation}.
\end{proof}
\section{Concluding remarks}
We characterized the images by the valuation of nonarchimedean spectrahedra which satisfy a certain genericity condition. Our results imply that the images of nongeneric spectrahedra are still closed semilinear sets. 
It is an open question to characterize the semilinear sets which arise in this way.
A special situation in which such a description is known in the nongeneric case
concerns tropical polyhedra. It relies on the Minkowski--Weyl theorem and does
not carry over to spectrahedra. 

\subsection*{Acknowledgement}
We thank the referees for their comments which led to improvements of this article.

\bibliographystyle{alpha}

\end{document}